\documentclass[11pt]{amsart}
\usepackage{amsmath,amssymb,latexsym,esint,cite,mathrsfs}
\usepackage{verbatim,wasysym}
\usepackage[left=2.6cm,right=2.6cm,top=3cm,bottom=3cm]{geometry}
\allowdisplaybreaks

\usepackage{color}
\usepackage{graphicx}
\usepackage{subfigure}

\usepackage{graphicx,epic,eepic}

\allowdisplaybreaks

\begin{document}

\title[Bubble tower solutions]
{Sign-changing bubble tower solutions for a Paneitz-type problem}

\author[W. Chen]{Wenjing Chen}
\address{\noindent W. Chen-School of Mathematics and Statistics, Southwest University,
Chongqing 400715, People's Republic of China}\email{wjchen@swu.edu.cn}

\author[X. Huang]{Xiaomeng Huang}
\address{\noindent X. Huang-School of Mathematics and Statistics, Southwest University,
Chongqing 400715, People's Republic of China.}\email{hhuangxiaomeng@126.com}

 \maketitle

\maketitle
\numberwithin{equation}{section}
\newtheorem{theorem}{Theorem}[section]
\newtheorem{lemma}[theorem]{Lemma}
\newtheorem{definition}[theorem]{Definition}
\newtheorem{proposition}[theorem]{Proposition}
\newtheorem{remark}[theorem]{Remark}
\allowdisplaybreaks

\maketitle

\noindent {\bf Abstract}: This paper is concerned with the following biharmonic problem
\begin{equation}\label{ineq}
\begin{cases}
\Delta^2 u=|u|^{\frac{8}{N-4}}u &\text{ in }  \ \Omega\backslash \overline{{B(\xi_0,\varepsilon)}},\\
  u=\Delta u=0 &\text{ on } \ \partial (\Omega \backslash \overline{{B(\xi_0,\varepsilon)}}),
\end{cases}
\end{equation}
where $\Omega$ is an open bounded domain in $\mathbb{R}^N$, $N\geq 5$, and $B(\xi_0,\varepsilon)$ is a ball centered at $\xi_0$ with radius $\varepsilon$, $\varepsilon$ is a small positive parameter.
We obtain the existence of solutions for problem (\ref{ineq}), which is an arbitrary large number of sign-changing solutions whose profile is a superposition of bubbles with alternate sign which concentrate at the center of the hole.

\vspace{3mm} \noindent {\bf Keywords}: Biharmonic equation; critical Sobolev exponent; sign-changing bubble tower solutions; reduction method.

\vspace{3mm}
\vspace{3mm}

\maketitle

\section{Introduction and statement of main result}

In this article, we consider the existence of sign-changing bubble tower solutions to the following problem
\begin{equation}\label{e1.1}
\begin{cases}
\Delta^2 u=|u|^{p-1}u &\text{ in }  \ \Omega_\varepsilon,\\
  u=\Delta u=0 &\text{ on } \ \partial \Omega_\varepsilon,
\end{cases}
\end{equation}
where $\Delta^2$ is the biharmonic operator, $\Omega_\varepsilon:= \Omega \backslash \overline{{B(\xi_0,\varepsilon)}}$ with $\Omega$ being an open bounded domain in $\mathbb{R}^N$, $N\geq 5$, $B(\xi_0,\varepsilon)$ is a ball centered at $\xi_0$ with radius $\varepsilon$, $\varepsilon>0$ small enough, and $p=\frac{N+4}{N-4}$ is the critical exponent in the sense that the embedding $H^2(\Omega)\cap H_0^1(\Omega)\hookrightarrow L^{p+1}(\Omega)$.

Problem (\ref{e1.1}) is related to the Paneitz operator, which is conformal operator of the fourth-order. It was first introduced by Paneitz \cite{Paneitz} for the study of smooth four-dimensional Riemannian manifolds, and Branson \cite{Branson1} generalized the dimension of the Riemannian manifold to $N$ dimension. Since this type of equation involving Paneitz operators is similar to geometric equations with Paneitz operators, it has received a lot of attention, we refer to \cite{Ala,Branson2,Ayed2,Chang1, Chang2, Chang3, Ebobisse, Mehdi, Gazzola} and references therein.

When the biharmonic operator in (\ref{e1.1}) is replaced by the Laplacian operator, consider the following problem
\begin{equation}\label{e1.2}
\begin{cases}
-\Delta u=|u|^{\frac{4}{N-2}}u &\text{ in }  \ \Omega,\\
  u=0 &\text{ on } \ \partial \Omega.
\end{cases}
\end{equation}
Solvability for problem \eqref{e1.2} is not a trivial issue, since it strongly depends on the geometry of $\Omega$.
A direct consequence of Pohozaev's identity \cite{ph} is that problem (\ref{e1.2}) has no positive solutions when the domain $\Omega$ is strictly star-shaped.
On the other hand, if $\Omega$ is an annulus, then Kazdan and Warner \cite{kw} showed that solvability for problem (\ref{e1.2}) is restored. Coron showed that symmetry is not really needed to have solvability in \cite{coron}, he obtained the existence of a positive solution to (\ref{e1.2}) in the case in which $\Omega$ has a small (not necessarily symmetric) hole.

The study of sign-changing solutions for elliptic problems with critical nonlinearity has received the interest of several authors in the last years, see for instance \cite{bar1,clm,vaira,vaira2} and references therein.
Here we focus our interest in existence and qualitative properties of sign-changing solutions to \eqref{e1.2} for domains $D$ which have a hole, that is in the Coron's setting.
The first result available in literature is the one contained in \cite{mupi2}, where a large number of sign changing solutions to  (\ref{e1.2}) in the presence of a single small hole has been proved. To be more precise, the authors assume that the domain is $\Omega\backslash B(0,\varepsilon)$, where $\Omega$ is a smooth bounded domain containing the origin, and it is symmetric with respect to the origin, while the hole is given by $B(0,\varepsilon)$, a round ball with radius $\varepsilon$.
Substantial improvement of this result was obtained in \cite{Ge} where the assumption of symmetry was removed,
see also \cite{mupi2,Musso}. Recently, the existence of a sequence of finite-energy, sign-changing solutions with a crown-like shape for problem (\ref{e1.2}) was obtained in \cite{demu}.

Consider the following biharmonic equation under the Navier boundary condition
\begin{align}\label{ek}
\left\{
  \begin{array}{ll}
\Delta^2 u=K|u|^{q-1}u &\text{ in }  \ \Omega,\\
  u=\Delta u=0 &\text{ on } \ \partial \Omega.
\end{array}
\right.
\end{align}
In the subcritical case, namely $q=\frac{N+4}{N-4}-\varepsilon$ with $\varepsilon>0$ small, when $K$ is a constant, the asymptotic behavior of solutions of (\ref{ek}) has been studied in \cite{Ayed3}. On the other hand,
Ayed and Ghoudi \cite{Ayed1} proved that the low energy sign-changing solutions to (\ref{ek}) that are close to two bubbles with different signs and they have to blow up either at two different points with the same speed or at a critical point of the Robin function. Yessine and Rabeh \cite{Yessine} constructed a solution with the shape of a tower of sign-changing bubbles as $\varepsilon$ goes to zero. When $K\neq 1$, Ghoudi\cite{Ghoudi} constructed sign-changing solutions of (\ref{ek}) having two bubbles and blowing up either at two different critical points of $K$ with the same speed or at the same critical point. See \cite{Chou, Mehdi} and the reference therein for the concentration phenomena of solutions of the subcritical problem (\ref{ek}).

Concerning the supercritical case, namely $q=\frac{N+4}{N-4}+\varepsilon$ with $\varepsilon>0$, Bouh\cite{Bouh} showed that there is no sign-changing solution with low energy which blow up at exactly two points for $\varepsilon$ small and proved that problem (\ref{e1.1}) has no bubble-tower sign-changing solutions. Ayed {\em et al.} \cite{Ayed3} got that the supercritical problem (\ref{e1.1}) has no solutions which concentrate around a point of $\Omega$ as $\varepsilon \rightarrow 0$.  The case $K$ is a nonconstant function, it was proved \cite{Bouh1} that for $\varepsilon$ small, (\ref{ek}) has no sign-changing solutions that blow up at two near points and also has no bubble-tower sign-changing solutions.

In the critical case $q=\frac{N+4}{N-4}$. Since the Sobolev embedding is not compact,  problem (\ref{e1.1}) is the lack of compactness. In fact, some researchers have got some results of the existence of solution. A first result of problem (\ref{e1.1}) was obtained by Van Der Vorst \cite{rcam}, who showed that when $\Omega_\varepsilon$ is a starshaped domain, (\ref{e1.1}) has no positive solutions. In \cite{Ebobisse}, Ebobisse and Ould Ahmedou investigated the influence of the topology of the domain $\Omega_\varepsilon$ on the existence of solution, they proved that (\ref{e1.1}) has a positive solution when some homology group of $\Omega_\varepsilon$ is nontrivial. Subsequently, in \cite{Gazzola}, Gazzola, Grunau and Squassina showed that this topological condition is sufficient, but not necessary by proving existence of nontrivial solutions in some contractible domains which are perturbations of small capacity of domains having nontrivial topology. When the domain $\Omega_\varepsilon$ is a bounded domain with a small ball removed, Alarc\'{o}n and Pistoia\cite{Ala}  constructed solutions of (\ref{e1.1}) blowing up at the center of the hole as the size of the hole goes to zero.

Moreover, as far as we know, there are few results on the sign-changing tower solutions for problem  (\ref{e1.1}).
Inspired by the above works, especially by \cite{Ala,Ge,mupi,Musso}, in the present paper, we construct sign-changing solutions  to problem (\ref{e1.1}), the shape of this solution is a superposition of bubbles with alternating sign centered at the center of the hole, the point $\xi_0$, as $\varepsilon$ goes to $0$.

Let us consider the following limit equation
\begin{equation}\label{e1.3}
\Delta^2 u=u^{\frac{N+4}{N-4}},\ \ \ u>0,\ \  \text{ in }  \mathbb{R}^N.
\end{equation}
Smooth radial solutions of \eqref{e1.3} are completely classified \cite{Lin} and are given by
\begin{equation}\label{limjie}
U_{\mu,\xi}(x)=\alpha_N\left(\frac{\mu}{\mu^2+|x-\xi|^2}\right)^{\frac{N-4}{2}},\ x\in \mathbb{R}^N,
\end{equation}
where $\alpha_N=(N(N-4)(N-2)(N+2))^{\frac{N-4}{8}}$, $\mu$ is a positive parameter and $\xi \in \mathbb{R}^N$.

Our main result can be stated as follows.

\begin{theorem}\label{thm1.1}
Assume $N\geq 5$. Then, given an integer $k\geq 1$, there exists $\varepsilon_0>0$ such that for any $\varepsilon \in (0,\varepsilon _0)$, there exists a pair of solutions $u_\varepsilon$ and $-u_\varepsilon$ to problem (\ref{e1.1}) such that
\begin{equation*}
u_\varepsilon(x)=\alpha_N\sum_{i=1}^k(-1)^{i+1}\left(\frac{d_i\varepsilon^{\frac{2i-1}{2k}}}{d_i^2\varepsilon^{2\frac{2i-1}{2k}}+|x-\xi_0|^2}\right)^{\frac{N-4}{2}}(1+o(1)),
\end{equation*}
where $d_1,\ldots,d_k$ are positive numbers depending only on $N$ and $k$ and $o(1)\rightarrow 0$ uniformly on compact subsets of $\Omega$ as $\varepsilon$ goes to $0$.
\end{theorem}

The proof of Theorem \ref{thm1.1} is based on the Lyapunov-Schmidt reduction method, we can refer to \cite{chen,cdg,Pino1, Pino2, Pino3,demu,Ge,mupi,Musso,Pistoia} and reference therein.
Let us point out the difficulties arising in the construction of the solution $u_\varepsilon$ in Theorem \ref{thm1.1}.
Taking into account that the solution concentrates around the hole and at the same time it must satisfy zero Navier boundary conditions on the boundary of the hole, it turns to be extremely delicate to study the behavior of the solution in the region around the hole, which was given in \cite{Ala}. Moreover, we have to do some fine estimates on interaction of different bubbles. The proof will also provide much finer estimates on the expansion of the energy functional, and  the $C^1$-estimate.

The paper is organized as follows. In Section \ref{prel}, we give some preliminary results.
The proof of the main result is given in Section \ref{mproof}. Section \ref{redu} is devoted to perform the finite dimensional reduction. Section \ref{exp} contains the asymptotic expansion of the reduced energy. Some technical estimates are given in Appendix.

\section{Preliminaries}\label{prel}

In this section, we build the first approximate solution of (\ref{e1.1}). In order to do this, let $P_\varepsilon$ be
the projection of $U_{\mu,\xi}$ on $H^2(\Omega_\varepsilon)\cap H_0^1(\Omega_\varepsilon)$, which satisfies
\begin{equation}\label{PU}
\begin{cases}
\Delta^2 P_\varepsilon U_{\mu,\xi}=U_{\mu,\xi}^p &\text{ in }  \ \Omega_\varepsilon,\\
  P_\varepsilon U_{\mu,\xi}=\Delta P_\varepsilon U_{\mu,\xi}=0 &\text{ on } \ \partial \Omega_\varepsilon,
\end{cases}
\end{equation}
where $U_{\mu,\xi}$ is given in (\ref{limjie}).
Let us denote by $G$ the Green's function of $\Delta^2$ in $\Omega$ under the Navier boundary condition and by $H$ its regular part, so that
\begin{equation*}
G(x,y)=\frac{1}{|x-y|^{N-4}}-H(x,y), \ \text{for}\ (x,y)\in \Omega\times\Omega.
\end{equation*}
The Robin's function $R:\Omega \rightarrow \mathbb{R}$ is defined as $R(x) = H(x, x)$.

By the comparison principle, we have the following crucial estimates.

\begin{lemma}\label{lem6.1}
Let $d>0$ small but fixed and $\xi=\xi_0+\mu \sigma$. If $\mu$ and $\sigma$ satisfy (\ref{e2.8}) and we define
\begin{equation}\label{e6.1}
R:=P_\varepsilon U_{\mu,\xi}-U_{\mu,\xi}+\alpha_N\mu^{\frac{N-4}{2}}H(x,\xi)+a_1 \varphi_1(\frac{x-\xi_0}{\varepsilon})+a_2\varphi_2(\frac{x-\xi_0}{\varepsilon}),
\end{equation}
where
\begin{equation}\label{e6.2}
a_1(\varepsilon,\mu,\sigma)=-\frac{\Delta U_{1,0}(\sigma)}{2(N-4)}\frac{\varepsilon^2}{\mu^{\frac{N}{2}}},\quad \qquad\qquad \varphi_1(x)=\frac{1}{|x|^{N-4}},
\end{equation}
\begin{equation}\label{e6.3}
a_2(\varepsilon,\mu,\sigma)=\frac{U_{1,0}(\sigma)}{\mu^{\frac{N-4}{2}}}+\frac{\Delta U_{1,0}(\sigma)}{2(N-4)}\frac{\varepsilon^2}{\mu^{\frac{N}{2}}},\quad \varphi_2(x)=\frac{1}{|x|^{N-2}},
\end{equation}
then for any $x\in \Omega_\varepsilon$, we have
\begin{equation}\label{e6.4}
|R|\leq c \left(\frac{\varepsilon^{N-1}}{\mu^{\frac{N+2}{2}}}\frac{1}{|x-\xi_0|^{N-4}}+\frac{\varepsilon^{N-1}}{\mu^{\frac{N-2}{2}}}\frac{1}{|x-\xi_0|^{N-2}}\right),
\end{equation}
\begin{equation}\label{e6.41}
|\partial_{\mu}R|\leq c \left(\frac{\varepsilon^{N-1}}{\mu^{\frac{N+4}{2}}}\frac{1}{|x-\xi_0|^{N-4}}+\frac{\varepsilon^{N-1}}{\mu^{\frac{N}{2}}}\frac{1}{|x-\xi_0|^{N-2}}\right),
\end{equation}
\begin{equation}\label{e6.42}
|\partial_{\sigma_i}R|\leq c \left(\frac{\varepsilon^{N-2}}{\mu^{\frac{N}{2}}}\frac{1}{|x-\xi_0|^{N-4}}+\frac{\varepsilon^{N-1}}{\mu^{\frac{N-2}{2}}}\frac{1}{|x-\xi_0|^{N-2}}\right)
\end{equation}
for some positive constants $c$.
\end{lemma}
\begin{proof}
We omit the proof of (\ref{e6.4}), which was given in \cite[Proposition 2.1]{Ala}, we are left to prove (\ref{e6.41}) and (\ref{e6.42}).
Let us denote by $\partial_{\mu}R=R_\mu(x)$ and define $\hat{R}_\mu(y)=\mu^{-\frac{N-6}{2}}R_\mu(\varepsilon y+\xi_0)$. Then function $\hat{R}_\mu(y)$ solves $\Delta^2 \hat{R}_\mu=0$ in $(\varepsilon^{-1}(\Omega-\xi_0)\setminus \overline{B_1(0)})$.

When $y\in \partial(\varepsilon^{-1}(\Omega-\xi_0))$, we have
\begin{align*}
\hat{R}_\mu(y)=&-\alpha_N\frac{N-4}{2}\left(\frac{|\varepsilon y-\mu \sigma|^2-\mu^2}{(\mu^2+|\varepsilon y-\mu \sigma|^2)^{\frac{N-2}{2}}}+2\mu\frac{(\varepsilon y-\mu \sigma, \sigma)}{(\mu^2+|\varepsilon y-\mu \sigma|^2)^{\frac{N-2}{2}}}\right)
\nonumber\\
&+\alpha_N\frac{N-4}{2}\left(\frac{1}{|\varepsilon y-\mu \sigma|^{N-4}}+\frac{2\mu(\varepsilon y-\mu \sigma, \sigma)}{|\varepsilon y-\mu \sigma|^{N-2}}\right)+\frac{N}{2}\frac{\Delta U_{1,0}(\sigma)}{2(N-4)}\frac{\varepsilon^2}{\mu^{N-2}}\frac{1}{|y|^{N-4}}
\nonumber\\
&-\left(\frac{N-4}{2}\frac{U_{1,0}(\sigma)}{\mu^{N-4}}+
\frac{N}{2}\frac{\Delta U_{1,0}(\sigma)}{2(N-4)}\frac{\varepsilon^2}{\mu^{N-2}}\right)\frac{1}{|y|^{N-2}},
\end{align*}
and
\begin{align*}
\Delta \hat{R}_\mu=&\frac{(N-4)\alpha_N\varepsilon^2}{2}\left[\frac{2(N-4)}{(\mu^2+|\varepsilon y-\mu \sigma|^2)^{\frac{N-2}{2}}}
+\frac{(N-2)\big((N-4)\mu^2+4\mu(\varepsilon y-\mu \sigma, \sigma)\big)}{(\mu^2+|\varepsilon y-\mu \sigma|^2)^{\frac{N}{2}}}\right.
\nonumber\\
&\qquad\qquad\qquad\ +\frac{2N(N-2)\mu^2\big(|\varepsilon y-\mu \sigma|^2+(\varepsilon y-\mu \sigma, \mu\sigma)\big)}{(\mu^2+|\varepsilon y-\mu \sigma|^2)^{\frac{N+2}{2}}}
-\frac{2(N-4)}{|\varepsilon y-\mu \sigma|^{N-2}}
\nonumber\\
&\qquad\qquad\qquad\ \left.-\frac{4(N-2)\mu(\varepsilon y-\mu \sigma, \sigma)}{|\varepsilon y-\mu \sigma|^N}\right]
-\frac{N}{2}\frac{\Delta U_{1,0}(\sigma)}{|y|^{N-2}}\frac{\varepsilon^{2}}{\mu^{N-2}}.
\end{align*}
Moreover, $|\hat{R}_\mu|=O(\mu^2+(\frac{\varepsilon}{\mu})^{N-2})$,  $|\Delta \hat{R}_\mu|=O(\varepsilon^2 (\mu^2+(\frac{\varepsilon}{\mu})^{N-2}))$ for $y\in \partial(\varepsilon^{-1}(\Omega-\xi_0))$.

For $y\in \partial B_1(0)$, we get
\begin{align*}
\hat{R}_\mu(y)=&-\alpha_N\frac{N-4}{2}\left(\frac{|\varepsilon y-\mu \sigma|^2-\mu^2}{(\mu^2+|\varepsilon y-\mu \sigma|^2)^{\frac{N-2}{2}}}+2\mu\frac{(\varepsilon y-\mu \sigma, \sigma)}{(\mu^2+|\varepsilon y-\mu \sigma|^2)^{\frac{N-2}{2}}}\right)
\nonumber\\
&+\alpha_N\frac{N-4}{2}\left(H(\varepsilon y+\xi_0, \xi)+\frac{2\mu}{N-4}(\nabla_{\xi}H(\varepsilon y+\xi_0, \xi), \sigma)\right)
-\frac{N-4}{2}\frac{U_{1,0}(\sigma)}{\mu^{N-4}},
\end{align*}
\begin{align*}
\Delta \hat{R}_\mu=&\frac{(N-4)\alpha_N\varepsilon^2}{2}\left[\frac{2(N-4)}{(\mu^2+|\varepsilon y-\mu \sigma|^2)^{\frac{N-2}{2}}}
+\frac{(N-2)\big((N-4)\mu^2+4\mu(\varepsilon y-\mu \sigma, \sigma)\big)}{(\mu^2+|\varepsilon y-\mu \sigma|^2)^{\frac{N}{2}}}\right.
\nonumber\\
&\qquad\qquad+\frac{2N(N-2)\mu^2\big(|\varepsilon y-\mu \sigma|^2+(\varepsilon y-\mu \sigma, \mu\sigma)\big)}{(\mu^2+|\varepsilon y-\mu \sigma|^2)^{\frac{N+2}{2}}}
-\Delta H(\varepsilon y+\xi_0, \xi)
\nonumber\\
&\qquad\qquad\left.+\frac{2\mu}{N-4}\Delta(\nabla_{\xi}H(\varepsilon y+\xi_0, \xi),\sigma)\right]
-\frac{N\Delta U_{1,0}(\sigma)}{2}\frac{\varepsilon^{2}}{\mu^{N-2}}.
\end{align*}
Moreover, $|\hat{R}_\mu|=O(\frac{\varepsilon}{\mu^{N-3}})$, $|\Delta \hat{R}_\mu|=O(\frac{\varepsilon^3}{\mu^{N-1}})$ for $ y\in \partial B_1(0)$.

Thus, by using a comparison argument, we get
\begin{align*}
|\hat{R}_\mu(y)|\leq c \left(\frac{\varepsilon^{3}}{\mu^{N-1}}\frac{1}{|y|^{N-4}}+\frac{\varepsilon}{\mu^{N-3}}\frac{1}{|y|^{N-2}}
+\varepsilon^2\mu^2|y|^2+\varepsilon\right).
\end{align*}
This implies that (\ref{e6.41}) holds.

Finally, let us denote by $\partial_{\sigma_i}R(x)=R_i(x)$ and define $\hat{R}_i(y)=\mu^{-\frac{N-2}{2}}R_i(\varepsilon y+\xi_0)$. Then,
\begin{align*}
\hat{R}_i(y)=&\alpha_N\frac{(N-4)}{2}\bigg(\frac{-2(\varepsilon y-\mu \sigma)_i}{(\mu^2+|\varepsilon y-\mu \sigma|^2)^{\frac{N-2}{2}}}-\frac{-2(\varepsilon y-\mu \sigma)_i}{|\varepsilon y-\mu \sigma|^{N-2}}\bigg)
\nonumber\\
&+\alpha_N\frac{(N-4)}{2}\bigg(\frac{4\sigma_i}{(1+|\sigma|^2)^{\frac{N}{2}}}-\frac{N\sigma_i(2|\sigma|^2+N)}{(1+|\sigma|^2)^{\frac{N+2}{2}}}  \bigg)\frac{\varepsilon^2}{(N-4)\mu^{N-1}}\frac{1}{|y|^{N-4}}
\nonumber\\
&+\alpha_N\frac{(N-4)}{2}\bigg(\frac{-2\sigma_i}{\mu^{N-3}(1+|\sigma|^2)^{\frac{N-2}{2}}}
\nonumber\\
&-\frac{1}{N-4}\frac{\varepsilon^2}{\mu^{N-1}}\bigg( \frac{4\sigma_i}{(1+|\sigma|^2)^{\frac{N}{2}}}-\frac{N\sigma_i(2|\sigma|^2+N)}{(1+|\sigma|^2)^{\frac{N+2}{2}}}\bigg) \bigg)\frac{1}{|y|^{N-2}}
\end{align*}
for $y\in \partial(\varepsilon^{-1}(\Omega-\xi_0))$, and
\begin{align*}
\hat{R}_i(y)=\alpha_N\frac{(N-4)}{2}\bigg(\frac{-2(\varepsilon y-\mu \sigma)_i}{(\mu^2+|\varepsilon y-\mu \sigma|^2)^{\frac{N-2}{2}}}+\frac{2}{N-4}\nabla_{\xi_i}H(\varepsilon y+\xi_0, \xi)-\frac{-2\sigma_i}{\mu^{N-3}(1+|\sigma|^2)^{\frac{N-2}{2}}}\bigg)
\end{align*}
for $y\in \partial B_1(0)$.
Moreover,
\begin{align*}
\Delta\hat{R}_i=&\alpha_N(N-4)\varepsilon^2\bigg(\frac{3(N-2)\Sigma_{i=1}^N(\varepsilon y-\mu \sigma)_i}{(\mu^2+|\varepsilon y-\mu \sigma|^2)^{\frac{N}{2}}}-\frac{N(N-2)\Sigma_{i=1}^N(\varepsilon y-\mu \sigma)_i^3}{(\mu^2+|\varepsilon y-\mu \sigma|^2)^{\frac{N+2}{2}}}
\nonumber\\
&-\frac{3(N-2)\Sigma_{i=1}^N(\varepsilon y-\mu \sigma)_i}{|\varepsilon y-\mu \sigma|^{N}}+\frac{N(N-2)\Sigma_{i=1}^N(\varepsilon y-\mu \sigma)_i^3}{|\varepsilon y-\mu \sigma|^{N+2}}
\nonumber\\
&-\frac{1}{\mu^{N-1}}\frac{1}{|y|^{N-2}}\bigg(\frac{4\sigma_i}{(1+|\sigma|^2)^{\frac{N}{2}}}
-\frac{N\sigma_i(2|\sigma|^2+N)}{(1+|\sigma|^2)^{\frac{N+2}{2}}}\bigg)\bigg)
\end{align*}
for $y\in \partial(\varepsilon^{-1}(\Omega-\xi_0))$, and
\begin{align*}
\Delta\hat{R}_i=&\alpha_N(N-4)\varepsilon^2\bigg(\frac{3(N-2)\Sigma_{i=1}^N(\varepsilon y-\mu \sigma)_i}{(\mu^2+|\varepsilon y-\mu \sigma|^2)^{\frac{N}{2}}}-\frac{N(N-2)\Sigma_{i=1}^N(\varepsilon y-\mu \sigma)_i^3}{(\mu^2+|\varepsilon y-\mu \sigma|^2)^{\frac{N+2}{2}}}
\nonumber\\
&+\frac{\Delta(\nabla_{\xi_i}H(\varepsilon y+\xi_0, \xi))}{N-4}
-\frac{1}{\mu^{N-1}}\bigg( \frac{4\sigma_i}{(1+|\sigma|^2)^{\frac{N}{2}}}
-\frac{N\sigma_i(2|\sigma|^2+N)}{(1+|\sigma|^2)^{\frac{N+2}{2}}}\bigg) \bigg)
\end{align*}
for $y\in \partial B_1(0).$
Therefore
\begin{align*}
|\hat{R}_i|=O(\frac{\varepsilon^{N-2}}{\mu^{N-1}}+\mu^2),\  |\Delta \hat{R}_i|=O(\varepsilon^2(\frac{\varepsilon^{N-2}}{\mu^{N-1}}+\mu^2))\   \text{ for } y\in \partial(\varepsilon^{-1}(\Omega-\xi_0)).
\end{align*}
\begin{align*}
|\hat{R}_i|=O(\frac{\varepsilon}{\mu^{N-2}}),\ |\Delta \hat{R}_i|=O(\frac{\varepsilon^2}{\mu^{N-1}})\   \text{ for } y\in \partial B_1(0).
\end{align*}
It is similar to the previous proof, we can get (\ref{e6.42}).
\end{proof}

The space $H^2(\Omega_\varepsilon)\cap H_0^1(\Omega_\varepsilon)$ is equipped with
the scalar product $\langle u, v\rangle=\int_{\Omega_\varepsilon}\Delta u \Delta v$, which induces the norm
\begin{equation*}
\|u\|=\left(\int_{\Omega_\varepsilon}|\Delta u|^2\right)^{\frac{1}{2}}.
\end{equation*}
If $u\in L^q(\Omega_\varepsilon)$, we denote by $|u|_q=(\int_{\Omega_\varepsilon}|u|^q)^{\frac{1}{q}}$ the $L^q $-norm.
By Sobolev embedding Theorem, there exists $C>0$, depending only on $N$, such that $|u|_{\frac{2N}{N-4}}\leq C\|u\|$ for all $u\in H^2(\Omega_\varepsilon)\cap H_0^1(\Omega_\varepsilon)$.
Consider now the adjoint operator of the embedding $i:H^2(\Omega_\varepsilon)\cap H_0^1(\Omega_\varepsilon)\hookrightarrow L^{\frac{2N}{N-4}}(\Omega_\varepsilon)$, namely the map
$i^*:L^{\frac{2N}{N+4}}(\Omega_\varepsilon)\hookrightarrow H^2(\Omega_\varepsilon)\cap H_0^1(\Omega_\varepsilon)$ is
defined as the (unique) weak solution of
\begin{equation*}
\Delta^2 u=\omega\ \ \text{in}  \ \Omega_\varepsilon,\qquad
u=\Delta u=0\ \ \text{on} \ \partial \Omega_\varepsilon.
\end{equation*}
Thus
\begin{equation*}
i^*(\omega)=u \Leftrightarrow \langle u,\varphi\rangle=\int_{\Omega_\varepsilon}\omega \varphi dx,  \quad  \forall \varphi \in H^2(\Omega_\varepsilon)\cap H_0^1(\Omega_\varepsilon).
\end{equation*}
Moreover, there exists a positive
constant $c$, which depends only on the dimension $N$, such that
\begin{equation}\label{e2.1}
\| i^*(\omega) \| \leq c|\omega |_{\frac{2N}{N+4}}  \quad \text{for\ all}\ \omega \in L^{\frac{2N}{N+4}}(\Omega_\varepsilon).
\end{equation}
Using the above definitions and notations, problem (\ref{e1.1}) can be rewritten as follows
\begin{equation}\label{e2.2}
u=i^*[f(u)],\quad u\in H^2(\Omega_\varepsilon)\cap H_0^1(\Omega_\varepsilon),
\end{equation}
where $f(u)=|u|^{p-1}u$ with $p=\frac{N+4}{N-4}$.

We next describe the shape of the solutions we are looking for.
Let $d>0$ be a small but fixed number. Given an integer $k$, let $\mu_{i}$, $i=1,2,\cdots,k$, be positive numbers and $\sigma_i$, $i=1,2,\cdots,k$, be points in $\mathbb{R}^N$ satisfying
\begin{equation}\label{e2.8}
d<\mu_{i}<\frac{1}{d} \ \text{and} \  |\sigma_i|<\frac{1}{d} \ \text{for} \  i=1,\ldots,k.
\end{equation}
We assume that
\begin{equation}\label{e2.6}
\mu_{i\varepsilon}=\mu_{i}\varepsilon^{\frac{2i-1}{2k} \theta}\  \text{with}\ \theta=\frac{2k(N-2)}{2k(N-2)-2},\ \ \ \text{and}\  \xi_{i\varepsilon}=\xi_0+\mu_{i\varepsilon}\sigma_i.
\end{equation}
We define
\begin{equation}\label{e2.5}
V(x)=\sum_{i=1}^k(-1)^{i+1}P_{\varepsilon}U_{\mu_{i\varepsilon},\xi_{i\varepsilon}}(x).
\end{equation}
We will use the compact notation $\mu=(\mu_1, \mu_2, \cdots, \mu_k)\in \mathbb{R}_+^k$ and $\sigma=(\sigma_1, \sigma_2, \cdots, \sigma_k)\in \mathbb{R}^{Nk}$.

We will look for the solutions of (\ref{e1.1}) or (\ref{e2.2}) with the form
\begin{equation}\label{e2.3}
u(x)=V(x)+\phi(x).
\end{equation}
Here the term $\phi$ is a smaller perturbation of $V$.

We next describe the term $\phi$ in (\ref{e2.3}). Let us recall (see \cite{Lu}) that every bounded solution to the linear equation
\begin{equation*}
\Delta^2\nu-f'(U_{\mu,\xi})\nu=0 \ \text{in} \ \mathbb{R}^N, \ \nu\in D^{2,2}(\mathbb{R}^N),
\end{equation*}
is a linear combination of the functions
\begin{equation*}
Z^0_{\mu,\xi}(x):=\frac{\partial U_{\mu,\xi}}{\partial \mu}=\alpha_N \left(\frac{N-4}{2}\right)\mu^{\frac{N-6}{2}}\frac{|x-\xi|^2-\mu^2}{(\mu^2+|x-\xi|^2)^{\frac{N-2}{2}}},
\end{equation*}
\begin{equation*}
Z^i_{\mu,\xi}(x):=\frac{\partial U_{\mu,\xi}}{\partial \xi_i}=\alpha_N (N-4)\mu^{\frac{N-4}{2}}\frac{x_i-\xi_i}{(\mu^2+|x-\xi|^2)^{\frac{N-2}{2}}}, \ i=1,\cdots,N.
\end{equation*}
We denote by $P_\varepsilon Z^i_{\mu,\xi}$ the projection of $Z^i_{\mu,\xi}$ onto $H^2(\Omega_\varepsilon)\cap H_0^1(\Omega_\varepsilon)$ and we define the subspace of $H^2(\Omega_\varepsilon)\cap H_0^1(\Omega_\varepsilon)$
\begin{equation*}
K:=\text{span}\{P_\varepsilon Z^i_{\mu_{j\varepsilon},\xi_{j\varepsilon}}:\ i=0,1,\cdots,N, j=1,\cdots,k \},
\end{equation*}
and
\begin{equation*}
K^\bot:=\{\phi \in H^2(\Omega_\varepsilon)\cap H_0^1(\Omega_\varepsilon): \langle \phi,P_\varepsilon Z^i_{\mu_{j\varepsilon},\xi_{j\varepsilon}} \rangle =0, \ i=0,1,\cdots,N, j=1,\cdots,k \}.
\end{equation*}
Let $\Pi: H^2(\Omega_\varepsilon)\cap H_0^1(\Omega_\varepsilon)\rightarrow K$ and $\Pi^\bot:H^2(\Omega_\varepsilon)\cap H_0^1(\Omega_\varepsilon)\rightarrow K^\bot$ be the orthogonal projections.
In order to solve problem (\ref{e2.2}), we will solve the couple of equations
\begin{equation}\label{e2.9}
\Pi^\bot \{V+\phi-i^*[f(V+\phi)]\}=0,
\end{equation}
\begin{equation}\label{e2.10}
\Pi \{V+\phi-i^*[f(V+\phi)]\}=0.
\end{equation}

\section{Scheme of the proof}\label{mproof}

We first give the following result, whose proof is postponed to Section \ref{redu} to solve equation (\ref{e2.9}).

\begin{proposition}\label{pro3.1}
For any $d>0$ small but fixed, there exist $\varepsilon_0>0$ and $c>0$ such that for any $\mu\in\mathbb{R}^k_{+}$ and $\sigma\in\mathbb{R}^{Nk}$ satisfying (\ref{e2.8}) and for any $\varepsilon \in(0,\varepsilon_0)$, there exists a unique solution $\phi=\phi(\mu,\sigma)$ which solves equation (\ref{e2.9}). Moreover
\begin{align}\label{e3.5}
\|\phi\|\leq c
\begin{cases}\varepsilon^{\frac{(N-4)\theta}{2k}\frac{p}{2}}, &\text{ if }N\geq 13,\\
  \varepsilon^{\frac{(N-4)\theta}{2k}} |\ln \varepsilon|, &\text{ if }N=12,\\
  \varepsilon^{\frac{(N-4)\theta}{2k}}, &\text{ if }5 \leq N \leq 11.\end{cases}
\end{align}
Finally, $(\mu,\sigma)\mapsto \phi(\mu,\sigma)$ is a $C^1$-map.
\end{proposition}

The energy functional associated to (\ref{e1.1}) is defined by
\begin{align}\label{e3.39}
J_\varepsilon (u)=\frac{1}{2}\int_{\Omega_\varepsilon}|\Delta u|^2dx-\frac{1}{p+1}\int_{\Omega_\varepsilon}|u|^{p+1}dx,\ \ \ \mbox{for}\ u\in H^2(\Omega_\varepsilon)\cap H_0^1(\Omega_\varepsilon) .
\end{align}
We define the function $I:\mathbb{R}_+^{k} \times \mathbb{R}^{Nk}\rightarrow  \mathbb{R}$ by
\begin{equation}\label{e3.40}
I(\mu,\sigma)=J_\varepsilon (V+\phi),
\end{equation}
where $V$ is defined in  (\ref{e2.5}) and the existence of $\phi$ is guaranteed by Proposition \ref{pro3.1}.

The next result, whose proof is postponed until Section \ref{exp}, allows us to solve equation (\ref{e2.10}), by reducing the problem to a finite dimensional one.

\begin{proposition}\label{pro3.4}
(i) If $(\mu,\sigma)$ is a critical point of $I$, then $V+\phi$ is a solution to (\ref{e2.2}), or equivalently of problem (\ref{e1.1}).

(ii) 
There holds
\begin{align}\label{e4.1}
I(\mu,\sigma)=&\frac{2kc_1}{N}\alpha_N^{p+1}+\frac{\alpha_N^{p+1}}{2}\Phi(\mu,\sigma)\varepsilon^{\frac{N-4}{2k}\theta}+o(\varepsilon^{\frac{N-4}{2k}\theta}),
\end{align}
as $\varepsilon\to0$, $C^1$-uniformly with respect to $\mu$ and $\sigma$ satisfying (\ref{e2.8}),
where
\begin{align}\label{e4.2}
\Phi(\mu,\sigma)=c_2H(\xi_0,\xi_0)\mu_1^{N-4}+c_3\frac{\Delta U_{1,0}(\sigma_k)U_{1,0}(\sigma_k)}{\mu_k^{N-2}}+2\sum_{l=1}^{k-1}\Gamma(\sigma_l)(\frac{\mu_{l+1}}{\mu_{l}})^{\frac{N-4}{2}},
\end{align}
and 
\begin{equation}\label{e4.17}
c_1= \int_{\mathbb{R}^N}\frac{1}{(1+|z|^2)^N},\quad c_2= \int_{\mathbb{R}^N}\frac{1}{(1+|z|^2)^{\frac{N+4}{2}}},\quad c_3=-\frac{3(N-2)meas(\mathbb{S}^{N-1})}{2\alpha_N^{p+1}},
\end{equation}
\begin{equation}\label{e4.20}
\Gamma (x)=\int_{\mathbb{R}^N}\frac{1}{(1+|y-x|^2)^{\frac{N+4}{2}}}\frac{1}{|y|^{N-4}}dy,\ \ x\in\mathbb{R}^N.
\end{equation}
\end{proposition}

Before giving the proof of Theorem \ref{thm1.1}, we first give the following lemma, which is used to prove the existence of non degenerate critical point of $I$.
\begin{lemma}\label{pro6.4}
There exists $\hat{\mu}\in \mathbb{R}_+^k$ such that $(\hat{\mu},0)$ is a non degenerate critical point of the function $\Phi$ introduced in (\ref{e4.2}).
\end{lemma}
\begin{proof}
For simplicity, we replace $\mu_i^{\frac{N-4}{2}}$ by $\nu_i$ and let
\begin{align*}
H_1:= c_2H(\xi_0,\xi_0),\ f(\sigma_k):=-\frac{3(N-2)}{2}meas(\mathbb{S}^{N-1})\frac{\Delta U_{1,0}(\sigma_k)U_{1,0}(\sigma_k)}{\alpha_N^{p+1}},\ g(\sigma_j):= 2\Gamma(\sigma_j).
\end{align*}
Then we can write  $\Phi=H_1\nu_1^2+\frac{f(\sigma_k)}{\nu_k^\frac{2(N-2)}{N-4}}+g(\sigma_1)\frac{\nu_2}{\nu_1}+g(\sigma_2)\frac{\nu_3}{\nu_2}+\ldots
+g(\sigma_{k-1})\frac{\nu_k}{\nu_{k-1}}.$
\par If $\sigma=0$, the quadratic form $\nu\mapsto H_1\nu_1^2$ is strictly positively definite, then $\Phi(\nu,0)$  has a minimum point $(\hat{\nu},0)$. Next, we claim that $(\hat{\nu},0)$ is a non degenerate critical point of $\Phi(\nu,\sigma)$.
Since
\begin{align*}
\mathscr{H}\Phi(\hat{\nu},0)=
\left(
  \begin{array}{cc}
    \mathscr{H}\Phi_\nu(\hat{\nu},0) & 0 \\
    0 & \mathscr{H}\Phi_{\sigma}(\hat{\nu},0) \\
  \end{array}
\right)
\end{align*}
it remains to prove $det(\mathscr{H}\Phi_\nu(\hat{\nu},0))\neq 0$ and $det(\mathscr{H}\Phi_{\sigma}(\hat{\nu},0))\neq 0$.

In fact, from \cite[Lemma 4.1]{Ge}, we have that $x=0$ is a non degenerate critical point of $g(x)$.
Since
\begin{align*}
&\frac{\partial^2(\Delta U_{1,0}(x)U_{1,0}(x))}{\partial x_i \partial x_j} \big|_{x=0}=0,\ \ \mbox{for}\ i\neq j;\\
&\frac{\partial^2(\Delta U_{1,0}(x)U_{1,0}(x))}{\partial  x_i^2 } \big|_{x=0}=\alpha_N^2(N-4)(2N^2-6N-4)\neq 0.
\end{align*}
We obtain $x=0$ is a non degenerate critical point of $f(x)$, namely, $det\big(\mathscr{H}\Phi_\sigma(\hat{\nu},0)\big)\neq 0.$

From $\nabla \Phi(\nu,\sigma)=0$, we have
\begin{align*}
&\frac{\partial \Phi}{\partial \nu_1}=2H_1\nu_1-g(\sigma_1)\frac{\nu_2}{\nu_1^2}=0;\\
&\frac{\partial \Phi}{\partial \nu_i}=\frac{g(\sigma_{i-1})}{\nu_{i-1}}-g(\sigma_i)\frac{\nu_{i+1}}{\nu_{i}^2}=0 \quad \text{for} \ i=2,...,k-1;\\
&\frac{\partial \Phi}{\partial \nu_k}=-\frac{2(N-2)}{N-4}\frac{f(\sigma_k)}{\nu_k^\frac{3N-8}{N-4}}+\frac{g(\sigma_{k-1})}{\nu_{k-1}}=0.
\end{align*}
Thus $2H_1\nu_1^2=g(\sigma_1)\frac{\nu_2}{\nu_1}=\cdot\cdot\cdot=g(\sigma_{k-1})\frac{\nu_{k}}{\nu_{k-1}}=\frac{2(N-2)}{N-4}\frac{f(\sigma_k)}
{\nu_k^\frac{2N-4}{N-4}}:= \lambda$.
This together with the fact that $\Phi({\nu},0)$ has a minimum point $(\hat{\nu},0)$, we find
\begin{align*}
\mathscr{H}\Phi_{{\nu}}(\hat{\nu},0)
&=\left(
  \begin{array}{cccccc}
2H_1+\frac{2g(0)\hat{\nu}_2}{\hat{\nu}_1^3}  &  -\frac{g(0)}{\hat{\nu}_1^2} & 0 & \cdots\ &0 &0\\
-\frac{g(0)}{\hat{\nu}_1^2}  &  \frac{2g(0)\hat{\nu}_3}{\hat{\nu}_2^3} & -\frac{g(0)}{\hat{\nu}_2^2} & \cdots\ & 0 & 0\\
0 & -\frac{g(0)}{\hat{\nu}_2^2}  & \frac{2g(0)\hat{\nu}_4}{\hat{\nu}_3^3} &\cdots\ & 0 & 0\\
\vdots  & \vdots & \vdots & \ddots & \vdots& \vdots\\
0 & 0  & 0 &\cdots\ & \frac{2g(0)\hat{\nu}_k}{\hat{\nu}_{k-1}^3} & -\frac{g(0)}{\hat{\nu}_{k-1}^2}\\
0 & 0  & 0 &\cdots\ & -\frac{g(0)}{\hat{\nu}_{k-1}^2} & \frac{2(N-2)(3N-8)f(0)}{(N-4)^2\hat{\nu}_k^\frac{4(N-3)}{N-4}}\\
\end{array}
\right)_{k\times k}
\nonumber\\
&\rightarrow
\left(
  \begin{array}{cccccc}
3\lambda  &  -g(0) & 0 & \cdots\ &0 &0\\
-\frac{\lambda^2}{g(0)}  &  2\lambda & -g(0) & \cdots\ & 0 & 0\\
0 & -\frac{\lambda^2}{g(0)}  & 2\lambda &\cdots\ & 0 & 0\\
\vdots  & \vdots & \vdots & \ddots& \vdots& \vdots \\
0 & 0  & 0 &\cdots\ & 2\lambda & -g(0)\\
0 & 0  & 0 &\cdots\ & -\frac{\lambda^2}{g(0)} & \frac{2(N-2)(3N-8)f(0)}{(N-4)^2\hat{\nu}_k^\frac{2N-4}{N-4}}\\
\end{array}
\right)_{k\times k}
:=Q.
\end{align*}
By calculations, we can get $det(Q)=\frac{2f(0)(N-2)(3N-8)}{(N-4)^2\hat{\nu}_k^\frac{2N-4}{N-4}}(2k-1)\lambda^{k-1}-(2k-3)\lambda^k
=\frac{4Nk-8k-4}{N-4}\lambda^{k}\neq 0$, it follows that $det\big(\mathscr{H}\Phi_{{\nu}}(\hat{\nu},0)\big)\neq 0$.
That proves our claim.
\end{proof}

\noindent{\it Proof of Theorem \ref{thm1.1}:}
In order to prove $V+\phi$ is the solution of problem (\ref{e1.1}), we need to find a critical point of the function $I$. From Proposition \ref{pro3.4}, we see that finding a critical point of $I$ is equivalent to that of $\Phi$.

By Lemma \ref{pro6.4}, there exists $\hat{\mu}\in \mathbb{R}_+^k$
such that $(\hat{\mu},0)$ is a non degenerate critical point of the function $\Phi$ defined in (\ref{e4.2}), which is stable with respect to $C^1$-perturbation.
Therefore, we deduce that the function $I$ has a critical point, denoted by $(\mu_\varepsilon,\sigma_\varepsilon)$, which satisfies that $\mu_\varepsilon\rightarrow \hat{\mu}$ and $\sigma_\varepsilon\rightarrow 0$ as $\varepsilon\rightarrow 0$.
\qed

\section{The finite dimensional reduction}\label{redu}

This section is devoted to the proof of Proposition \ref{pro3.1}. Let us introduce the linear operator $L: K^\perp \rightarrow K^ \perp$ defined by
\begin{align}\label{e3.1}
L(\phi):=\Pi^\perp \left\{\phi-i^*[f'(V)\phi]\right\}.
\end{align}
We first study the invertibility of the operator $L$.

\begin{lemma}\label{lem3.2}
For any $d>0$ small but fixed, there exist $\varepsilon_0>0$ and $c>0$ such that for any $\mu\in\mathbb{R}^k_{+}$ and $\sigma\in\mathbb{R}^{Nk}$ satisfying (\ref{e2.8}) and for any $\varepsilon\in (0, \varepsilon_0)$, we have
\begin{align*}
\| L(\phi)\| \geq c\|\phi\| \quad  \text{ for any } \phi\in K^\perp.
\end{align*}
\end{lemma}
\begin{proof}
We argue by contradiction. Assume there exist $d>0$ and a sequence $\varepsilon_n\rightarrow 0$,  $\sigma_n \in {\mathbb{R}}^{Nk} $ and
$\mu_n \in {\mathbb{R}}^k_{+}$, with $\sigma_{n,i}\rightarrow \sigma_i\in{\mathbb{R}}^{N}$ and $\mu_{n,i}\rightarrow \mu_{i}>0$ for $i=1,\cdots,k$,  $\phi_n$, $\varsigma_n \in K^\perp$ such that
\begin{align}\label{e3.6}
 L(\phi_n)=\varsigma_n \text{ in }\Omega_{\varepsilon_n}\ \ \mbox{with}\
  \| \varsigma_n\|\rightarrow 0, \ \ \mbox{and}\
  \| \phi_n\|=1 \text{ as }n\rightarrow \infty.
\end{align}
By (\ref{e3.6}), there exists $\omega_n \in K$ such that
\begin{align}\label{e3.7}
\phi_n-i^*[f'(V)\phi_n]=\varsigma_n+\omega_n.
\end{align}

{\it Step 1}. We claim that
\begin{align}\label{e3.8}
\| \omega_n\|\rightarrow 0\ \text{as}\ n\rightarrow \infty.
\end{align}
In fact, Let $\omega_n=\sum_{\stackrel{j=0,\cdots,N,}{i=1,\cdots,k}}b_n^{ij}PZ_i^j$,
where $PZ_i^j:=P_{\varepsilon}Z_{\mu_{n,i},\xi_{n,i}}^j$.
For $l=1,2,\cdots,k$ and $h=0,1,\cdots,N$, we multiply (\ref{e3.7}) by $PZ_l^h$, and taking
into account that  $\phi_n, \varsigma_n \in K^\perp$, we get
\begin{align}\label{e3.9}
\langle \omega_n, PZ_l^h \rangle=-\int_{\Omega_{\varepsilon_n}}f'(V)\phi_nPZ_l^h dx.
\end{align}
From Lemma \ref{lem6.3}, we have
\begin{align}\label{e3.10}
\langle \omega_n, PZ_l^h \rangle=b_n^{lh}\left[c_h\left(\frac{1}{{\mu_{n,l}^{2}}}\right)+o\left(\frac{1}{{\mu_{n,l}^{2}}}\right)\right]
+o\left(\frac{1}{{\mu_{n,l}^{2}}}\right)\Big[\sum_{\stackrel{j=0,\cdots,N,}{j\neq h}}b_n^{lj}+\sum_{\stackrel{j=0,\cdots,N,}{i=1,\cdots,k,i\neq l}}b_n^{ij}\Big].
\end{align}
Since $\langle PZ_l^h, \phi_n \rangle=0$ for $\phi_n\in K^\perp$, we obtain
\begin{align}\label{e3.11}
\int_{\Omega_{\varepsilon_n}}f'(V)\phi_nPZ_l^h=&\int_{\Omega_{\varepsilon_n}}\bigg(f'(V)\phi_n(PZ_l^h-Z_l^h)+\left[f'(V)-pU_l^{p-1}\right]\phi_nZ_l^h
+pU_l^{p-1}\phi_nZ_l^h\bigg)
\nonumber\\
\leq &|f'(V)|_{\frac{N}{4}}|\phi_n|_{\frac{2N}{N-4}}|PZ_l^h-Z_l^h|_{\frac{2N}{N-4}}
+|f'(V)-pU_l^{p-1}|_{\frac{N}{4}}|\phi_n|_{\frac{2N}{N-4}}|Z_l^h|_{\frac{2N}{N-4}}
\nonumber\\
=&o\left(\frac{1}{{\mu_{n,l}}}\right).
\end{align}
By (\ref{e3.9})-(\ref{e3.11}), we get (\ref{e3.8}).

{\it Step 2}. Let us define
\begin{align*}
u_n:=\phi_n-\varsigma_n-\omega_n.
\end{align*}
We claim that
\begin{align}\label{e3.12}
\liminf_{n\rightarrow \infty}\int_{\Omega_{\varepsilon_n}}f'(V)u_n^2=c^2>0.
\end{align}
By (\ref{e3.7}), we get $u_n=i^*[f'(V)\phi_n]=i^*[f'(V)(u_n+\varsigma_n+\omega_n)]$,
then
\begin{align}\label{e3.14}
\left\{
\begin{array}{rcl}
&\Delta^2u_n=f'(V)u_n+f'(V)(\varsigma_n+\omega_n) \text{ in } \Omega_{\varepsilon_n},\\
&u_n=\Delta u_n=0 \text{ on } \partial \Omega_{\varepsilon_n}.
\end{array} \right.
\end{align}
We multiply the first equation in (\ref{e3.14}) by $u_n$ and integrate in $\Omega_{\varepsilon_n}$, we obtain
\begin{align}\label{e3.15}
\|u_n\|^2=\int_{\Omega_{\varepsilon_n}}f'(V)u_n^2+\int_{\Omega_{\varepsilon_n}}f'(V)(\varsigma_n+\omega_n)u_n.
\end{align}
Since $\|\phi_n\|\rightarrow 1$, $\|\varsigma_n\|\rightarrow 0$ and $\|\omega_n\|\rightarrow 0$, we see that $\|u_n\|\rightarrow 1$ and
\begin{align}\label{e3.16}
\left|\int_{\Omega_{\varepsilon_n}}f'(V)(\varsigma_n+\omega_n)u_n\right|\leq & |f'(V)|_{\frac{N}{4}}|\varsigma_n+\omega_n|_{\frac{2N}{N-4}}|u_n|_{\frac{2N}{N-4}}
\leq C\|\varsigma_n+\omega_n\|=o(1)\|u_n\|.
\end{align}
Then (\ref{e3.12}) follows from (\ref{e3.15}) and (\ref{e3.16}).

{\it Step 3}. We set
\begin{align}\label{e3.19}
A_{n,l}:= B(\xi_0,\sqrt{\mu_{n,l}\mu_{n,l-1}})\backslash B(\xi_0,\sqrt{\mu_{n,l}\mu_{n,l+1}})
\end{align}
with $\mu_{n,0}\mu_{n,1}=r^2$ for some $r>0$ and $\mu_{n,k}\mu_{n,k+1}=\varepsilon^2$.

For any $j=1,2\cdots k$, let $\chi_{n,j}$  be a smooth cut-off function such that
\begin{align}\label{e3.18}
\begin{cases}\chi_{n,j}(x)=1, \text{ if }\sqrt{\mu_{n,j}\mu_{n,j+1}}\leq|x-\xi_0|\leq \sqrt{\mu_{n,j}\mu_{n,j-1}},\\
  \chi_{n,j}(x)=0, \text{ if }|x-\xi_0|\leq\frac{\sqrt{\mu_{n,j}\mu_{n,j+1}}}{2} \text{ or } |x-\xi_0|\geq 2\sqrt{\mu_{n,j}\mu_{n,j-1}},\\
  |\nabla \chi_{n,j}(x)|\leq C_1 \text{ and }|\Delta \chi_{n,j}(x)|\leq \frac{C_2}{\sqrt{\mu_{n,j}\mu_{n,j-1}}},\quad with \quad C_1>0, \quad  C_2 >0 .\end{cases}
\end{align}

We define $\check{u}_{n,j}(y):={\mu}_{n,j}^{\frac{N-4}{2}}{u}_{n}({\mu}_{n,j}y+\xi_0)\chi_{n,j}({\mu}_{n,j}y+\xi_0)$.
We will show that, for any $j=1,2\cdots k$,
\begin{align}\label{e3.17}
\check{u}_{n,j}\rightharpoonup 0 \ \text{ in} \ D^{2,2}(\mathbb{R}^N) \text{ and }
  \check{u}_{n,j}\rightarrow 0 \ \text{ in} \ L^q_{loc}(\mathbb{R}^N) \ \text{for} \ q  \in [2,2^*). &
\end{align}
Let $x={\mu}_{n,j}y+\xi_0$, then
\begin{align}\label{e3.20}
\nabla \check{u}_{n,j}(y)={\mu}_{n,j}^{\frac{N-2}{2}}\left[\nabla {u}_{n}(x)\chi_{n,j}(x)+{u}_{n}(x)\nabla\chi_{n,j}(x)\right],
\end{align}
\begin{align}\label{e3.21}
\Delta \check{u}_{n,j}(y)={\mu}_{n,j}^{\frac{N}{2}}\left[\Delta {u}_{n}(x)\chi_{n,j}(x)+2\nabla{u}_{n}(x)\nabla\chi_{n,j}(x)+{u}_{n}(x)\Delta\chi_{n,j}(x)\right],
\end{align}
and
\begin{align}\label{e3.23}
\Delta^2\check{u}_{n,j}(y)=&{\mu}_{n,j}^{\frac{N+4}{2}}\bigg[\Delta^2{u}_{n}(x)\chi_{n,j}(x)+4\nabla \Delta{u}_{n}(x)\nabla \chi_{n,j}(x)+6\Delta{u}_{n}(x)\Delta \chi_{n,j}(x)
\nonumber\\
&\qquad\quad+4\nabla{u}_{n}(x)\nabla \Delta \chi_{n,j}(x)+{u}_{n}(x)\Delta^2 \chi_{n,j}(x)\bigg].
\end{align}
Then from (\ref{e3.18}), (\ref{e3.21}) and $\|{u}_{n}\|\rightarrow 1$, we have
\begin{align*}
\int_{\mathbb{R}^N} |\Delta \check{u}_{n,j}|^{2}&\leq \mu_{n,j}^{N}\int_{\frac{\sqrt{\mu_{n,j}\mu_{n,j+1}}}{2}\leq|x-\xi_0|\leq 2\sqrt{\mu_{n,j}\mu_{n,j-1}}}|\Delta u_n|^2+u_{n,j}^2(\Delta\chi_{n,j})^2+2|\nabla u_n\nabla \chi_{n,j}|^2
\leq C.
\end{align*}
Thus up to a subsequence,  $\check{u}_{n,j}\to \check{u}_{j}$ weakly in $D^{2,2}(\mathbb{R}^N)$ and strongly in $L^q_{loc}(\mathbb{R}^N)$ for any $q \in [2,2^*)$.

We claim that $\check{u}_{j}$ solves the problem
\begin{align}\label{e3.27}
\Delta^2\check{u}_{j}=f'(U_{1,\sigma_j})\check{u}_{j} \quad \text{in} \ \mathbb{R}^N.
\end{align}
and satisfies the orthogonality conditions
\begin{align}\label{e3.28}
\int_{\mathbb{R}^N}\Delta Z^h_{1,\sigma_j}\Delta \check{u}_{j}=0,\quad h=0,1,\cdots,N.
\end{align}
These two facts imply that $\check{u}_{j}=0$, namely, (\ref{e3.17}) holds.

We are thus led to prove (\ref{e3.27}) and (\ref{e3.28}). We start with (\ref{e3.27}).

By (\ref{e3.23}) and (\ref{e3.14}), we get
\begin{align*}
\Delta^2\check{u}_{j}(y)
=&{\mu}_{n,j}^{\frac{N+4}{2}}\bigg(f'(V({\mu}_{n,j}y+\xi_0)){u}_{n}({\mu}_{n,j}y+\xi_0)\chi_{n,j}({\mu}_{n,j}y+\xi_0)
\nonumber\\
&+f'(V({\mu}_{n,j}y+\xi_0))(\varsigma_n({\mu}_{n,j}y+\xi_0)+\omega_n({\mu}_{n,j}y+\xi_0))\chi_{n,j}({\mu}_{n,j}y+\xi_0)
\nonumber\\
&+4\nabla \Delta{u}_{n}({\mu}_{n,j}y+\xi_0)\nabla \chi_{n,j}({\mu}_{n,j}y+\xi_0)+6\Delta{u}_{n}({\mu}_{n,j}y+\xi_0)\Delta \chi_{n,j}({\mu}_{n,j}y+\xi_0)
\nonumber\\
&+4\nabla{u}_{n}({\mu}_{n,j}y+\xi_0)\nabla \Delta \chi_{n,j}({\mu}_{n,j}y+\xi_0)+{u}_{n}({\mu}_{n,j}y+\xi_0)\Delta^2 \chi_{n,j}({\mu}_{n,j}y+\xi_0)\bigg).
\end{align*}
Then, for any $\varphi\in C_0^{\infty}(\mathbb{R}^N)$, we have
\begin{align}\label{e3.24}
&\int_{\mathbb{R}^N}\Delta \check{u}_{n,j}(y) \Delta \varphi(y)
={\mu}_{n,j}^4\int_{\mathbb{R}^N}f'(V({\mu}_{n,j}y+\xi_0))\check{u}_{n,j}(y)\varphi(y)
\nonumber\\
&+{\mu}_{n,j}^{\frac{N+4}{2}}\int_{\mathbb{R}^N}f'(V({\mu}_{n,j}y+\xi_0))(\varsigma_n({\mu}_{n,j}y+\xi_0)+\omega_n({\mu}_{n,j}y+\xi_0))\chi_{n,j}({\mu}_{n,j}y+\xi_0)\varphi(y)
\nonumber\\
&+4{\mu}_{n,j}^{\frac{N+4}{2}}\int_{\mathbb{R}^N}\nabla \Delta{u}_{n}({\mu}_{n,j}y+\xi_0)\nabla \chi_{n,j}({\mu}_{n,j}y+\xi_0)\varphi(y)
\nonumber\\
&+6{\mu}_{n,j}^{\frac{N+4}{2}}\int_{\mathbb{R}^N}\Delta{u}_{n}({\mu}_{n,j}y+\xi_0)\Delta \chi_{n,j}({\mu}_{n,j}y+\xi_0)\varphi(y)
\nonumber\\
&+4{\mu}_{n,j}^{\frac{N+4}{2}}\int_{\mathbb{R}^N}\nabla{u}_{n}({\mu}_{n,j}y+\xi_0)\nabla \Delta \chi_{n,j}({\mu}_{n,j}y+\xi_0)\varphi(y)
\nonumber\\
&+{\mu}_{n,j}^{\frac{N+4}{2}}\int_{\mathbb{R}^N}{u}_{n}({\mu}_{n,j}y+\xi_0)\Delta^2 \chi_{n,j}({\mu}_{n,j}y+\xi_0)\varphi(y)
:=L_1+L_2+L_3+L_4+L_5+L_6.
\end{align}
If $\frac{\sqrt{\mu_{n,j}\mu_{n,j+1}}}{2} \leq|\mu_{n,j}y|\leq 2{\sqrt{\mu_{n,j}\mu_{n,j-1}}}$, we get
\begin{align}\label{e3.25}
f'(V({\mu}_{n,j}y+\xi_0))=f'\Big(\frac{1}{{\mu}_{n,j}^{\frac{N-4}{2}}}U_{1,0}(y-\sigma_j)+\sum_{\stackrel{i=1\cdots k}{i\neq j}}U_{n,i}({\mu}_{n,j}y+\xi_0)+o(1)\Big),
\end{align}
and
\begin{align}\label{e3.26}
U_{n,i}({\mu}_{n,j}y+\xi_0)=
\begin{cases}O\left(\frac{1}{{\mu}_{n,i}^{\frac{N-4}{2}}}\right), \qquad\qquad\quad\ \text{ if }i<j,\\[5mm]
  O\left(\frac{{\mu}_{n,i}^{\frac{N-4}{2}}}{{\mu}_{n,j}^{N-4}}\frac{1}{|y|^{N-4}}\right), \qquad\quad\text{ if }i>j.\end{cases}
\end{align}
By $\text{Lebesgue's}$ dominated convergence Theorem, we get
\begin{align*}
L_1\rightarrow \int_{\mathbb{R}^N}f'(U_{1,0}(y-\sigma_j))\check{u}_{j}(y)\varphi(y).
\end{align*}
Using H\"{o}lder inequality, we have
\begin{equation*}
|L_2|\leq  c{\mu}_{n,j}^{\frac{N+4}{2}}|f'(V({\mu}_{n,j}y+\xi_0))|_{\frac{N}{4}}|\varsigma_n({\mu}_{n,j}y+\xi_0)+\omega_n({\mu}_{n,j}y+\xi_0)|_{\frac{2N}{N-4}}|\chi_{n,j}
({\mu}_{n,j}y+\xi_0)|_{\frac{2N}{N-4}}
=o(1).
\end{equation*}
We can also see that $L_3$, $L_4$, $L_5$, $L_6\rightarrow 0$ in the same way.
Therefore, (\ref{e3.27}) follows by passing to the limit in (\ref{e3.24}).

Let us now prove (\ref{e3.28}). We have
\begin{align}\label{e3.33}
&\int_{\mathbb{R}^N}\Delta Z^h_{1,\sigma_j}(y)\Delta \check{u}_{n,j}(y)dy
\nonumber\\
=&\mu_{n,j}\int_{\frac{\sqrt{\mu_{n,j}\mu_{n,j+1}}}{2} \leq|x-\xi_0|\leq 2{\sqrt{\mu_{n,j}\mu_{n,j-1}}}}f'(U_{\mu_{n,j},\xi_{n,j}}(x)) Z^h_{\mu_{n,j},\xi_{n,j}}(x){u}_{n}(x)\chi_{n,j}(x)dx
\nonumber\\
=&\mu_{n,j}\left[\int_{A_{n,j}}f'(U_{\mu_{n,j},\xi_{n,j}}(x)) Z^h_{\mu_{n,j},\xi_{n,j}}(x){u}_{n}(x)+o(1)\right].
\end{align}
Now we observe that, by (\ref{e3.9}) and (\ref{e3.11})
\begin{align}\label{e3.29}
\mu_{n,j}\int_{\Omega_{\varepsilon_n}}\Delta PZ^h_{\mu_{n,j},\xi_{n,j}}(x)\Delta \check{u}_{n,j}(x)
=-\mu_{n,j}\langle PZ^h_{\mu_{n,j},\xi_{n,j}}, \omega_n \rangle=o(1).
\end{align}
On the other hand
\begin{align}\label{e3.30}
&\mu_{n,j}\int_{\Omega_{\varepsilon_n}}\Delta PZ^h_{\mu_{n,j},\xi_{n,j}}(x)\Delta \check{u}_{n,j}(x)
\nonumber\\
=&\mu_{n,j}\Big[\int_{A_{n,j}}+\sum_{l\neq j}\int_{A_{n,l}}+\int_{\Omega_{\varepsilon_n}\backslash B(\xi_{n,j},r)}\Big]f'(U_{\mu_{n,j},\xi_{n,j}}(x))Z^h_{\mu_{n,j},\xi_{n,j}}(x){u}_{n}(x)\nonumber\\
=&\mu_{n,j}\int_{A_{n,j}}f'(U_{\mu_{n,j},\xi_{n,j}}(x))Z^h_{\mu_{n,j},\xi_{n,j}}(x){u}_{n}(x)+o(1),
\end{align}
since
\begin{align*}
&\left|\mu_{n,j}\int_{\Omega_{\varepsilon_n}\backslash B(\xi_{n,j},r)}f'(U_{\mu_{n,j},\xi_{n,j}}(x))Z^h_{\mu_{n,j},\xi_{n,j}}(x){u}_{n}(x)\right|
\nonumber\\
\leq &c\mu_{n,j}|f'(U_{\mu_{n,j},\xi_{n,j}})|_{\frac{N}{4}}|Z^h_{\mu_{n,j},\xi_{n,j}}|_{\frac{2N}{N-4}}|{u}_{n}|_{\frac{2N}{N-4}}
=o(1),
\end{align*}
and for $l\neq j$,
\begin{align*}
&\left|\mu_{n,j}\int_{A_{n,l}}f'(U_{\mu_{n,j},\xi_{n,j}}(x))Z^h_{\mu_{n,j},\xi_{n,j}}(x){u}_{n}(x)\right|\nonumber\\
\leq & c\mu_{n,j}|f'(U_{\mu_{n,j},\xi_{n,j}})|_{\frac{N}{4}}|Z^h_{\mu_{n,j},\xi_{n,j}}|_{\frac{2N}{N-4}}|{u}_{n}|_{\frac{2N}{N-4}}
=o(1).
\end{align*}
By (\ref{e3.33}), (\ref{e3.29}) and (\ref{e3.30}), we get (\ref{e3.28}).

{\it Step 4}. We show that a contradiction arises with (\ref{e3.12}) , by showing that
\begin{align}\label{e3.34}
\int_{\Omega_{\varepsilon_n}}f'(V_{n})u_n^2=o(1).
\end{align}
This fact concludes the proof of this lemma.

Let us prove (\ref{e3.34}). We have
$$\int_{\Omega_{\varepsilon_n}}f'(V_{n})u_n^2=\sum_{j=1}^k\int_{A_{n,j}}f'(V_{n})u_n^2+\int_{\Omega_{\varepsilon_n}\setminus B(\xi_{0},r)}f'(V_{n})u_n^2.$$
For any $j=1,\cdots,k$, let $x-\xi_{0}=\mu_{n,j}y$, then we have
\begin{align*}
 &\int_{A_{n,j}}f'(V_{n})u_n^2 dx\leq c\sum_{i=1}^k\int_{A_{n,j}}U_{\mu_{n,i},\xi_{n,i}}^{p-1}u_n^2 dx
\nonumber\\
\leq &c\sum_{i=1}^k\mu_{n,j}^4\int_{\mathbb{R}^N}\left(\frac{\mu_{n,i}}{\mu_{n,i}^2+\mu_{n,j}^2|y-\frac{\mu_{n,i}}{\mu_{n,j}}\sigma_i|^2}\right)^4\check{u}_{n,j}^2(y) dy
\nonumber\\
\leq &c\sum_{i>j}\left(\frac{\mu_{n,i}}{\mu_{n,j}}\right)^4+c\int_{\mathbb{R}^N}\left(\frac{1}{1+|y|^2}\right)^4\check{u}_{n,j}^2(y) dy+c\sum_{i<j}\left(\frac{\mu_{n,j}}{\mu_{n,i}}\right)^4+o(1)
=o(1),
\end{align*}
and $\int_{\Omega_{\varepsilon_n}\setminus B(\xi_{0},r)}f'(V_{n})u_n^2=o(1)$.
That concludes the proof.
\end{proof}

Next, we study the nonlinear problem. We see that equation (\ref{e2.9}) is equivalent to
\begin{align}\label{e3.4}
L(\phi)=N(\phi)+R,
\end{align}
where the nonlinear term $N: K^\perp \rightarrow K^ \perp$ is defined by
\begin{align}\label{e3.2}
N(\phi):=\Pi^\perp \left\{i^*[f(V+\phi)-f(V)-f'(V)\phi]\right\},
\end{align}
the error term is defined by
\begin{align}\label{e3.3}
R:=\Pi^\perp \left\{i^*[f(V)]-V\right\}.
\end{align}

Based on Lemma \ref{lem3.2}, we can get the following result.
\begin{proposition}\label{pro3.3}
For any $d>0$ small but fixed, there exists $\varepsilon_0>0$ and $c>0$ such that for any $\mu\in\mathbb{R}^k_{+}$ and $\sigma\in\mathbb{R}^{Nk}$ satisfying(\ref{e2.8}) and for any $h\in K^\perp$, for any $\varepsilon \in(0,\varepsilon_0)$, there exists a unique solution $\phi \in K^\perp$ to $L(\phi)=h$. Furthermore,
\begin{align}\label{e3.35}
\|h\|\geq c\|\phi \|.
\end{align}
\end{proposition}
\begin{proof}
By Lemma \ref{lem3.2}, we get $L^{-1}(h)=\phi-i^*[f'(V)\phi]$.
Let $\hat{h}=L^{-1}(h)$ and $G(\phi)=-i^*[f'(V)\phi]$, then
\begin{align}\label{e3.36}
\phi+G(\phi)=\hat{h},
\end{align}
where $G:K^\perp\rightarrow K^\perp$ is a linear operator.
We next prove that $G$ is a compact operator.
\par Assume $\phi_n\rightharpoonup 0$ in $H^2(\Omega_\varepsilon)\cap H_0^1(\Omega_\varepsilon)$, then $\phi_n\rightarrow 0$ in $L^2$ over compacts and
\begin{align*}
|\langle G(\phi_n), \varphi \rangle|=&|\int_{\Omega_\varepsilon}f'(V)\phi_n\varphi|
\leq \left(\int_{\Omega_\varepsilon}pV^{p-1}\phi_n^2\right)^{\frac{1}{2}}\left(\int_{\Omega_\varepsilon}pV^{p-1}\varphi^2\right)^{\frac{1}{2}}
\leq c \left(\int_{\Omega_\varepsilon}pV^{p-1}\phi_n^2\right)^{\frac{1}{2}}\|\varphi\|.
\end{align*}
Let $\varphi=G(\phi_n)$, then $||G(\phi_n)|| \leq c \left(\int_{\Omega_\varepsilon}pV^{p-1}\phi_n^2\right)^{\frac{1}{2}}\rightarrow 0$.
Consequently, $G$ is a compact operator. By $\text{Fredholm's}$ alternative, we get equation (\ref{e3.36}) has a unique solution for each $\hat{h}$. The estimate (\ref{e3.35}) follows from Lemma \ref{lem3.2}.
\end{proof}

Next, we want to study the estimate of the error term $R$ defined in (\ref{e3.3}).
\begin{lemma}\label{lem4.4}
For any $d>0$ small but fixed, for any $\mu\in\mathbb{R}^k_{+}$ and $\sigma\in\mathbb{R}^{Nk}$ satisfying (\ref{e2.8}), if $\varepsilon>0$ is small enough, there holds
\begin{align*}
\|R\|\leq c
\begin{cases}\varepsilon^{ \frac{(N-4)\theta}{2k}\frac{p}{2}}, &\text{ if }N\geq 13,\\
  \varepsilon^{ \frac{(N-4)\theta}{2k}}|\ln \varepsilon|, &\text{ if }N=12,\\
  \varepsilon^{ \frac{(N-4)\theta}{2k}}, &\text{ if }5\leq N\leq 11.\end{cases}
\end{align*}
\end{lemma}
\begin{proof}
For simplicity, we write $U_{j}:=U_{\mu_{j\varepsilon},\xi_{j\varepsilon}}$, $PU_{j}:=P_{\varepsilon}U_{\mu_{j\varepsilon},\xi_{j\varepsilon}}$. By the definition of $i^*$ and $\Delta^2 PU_j=U_j^p$, we have $PU_j=i^*(U_j^p)=i^*(f(U_j))$ and
\begin{align*}
R=\Pi^\perp\Big\{i^*[f(V)]-\sum_{j=1}^k (-1)^{j+1}PU_j\Big\}=\Pi^\perp \Big\{i^*[f(V)-\sum_{j=1}^k(-1)^{j+1}f(U_j)]\Big\}.
\end{align*}
Using (\ref{e2.1}), we find
\begin{align}\label{rdfg}
\|R\|
\leq &C\Big(|f(V)-\sum_{j=1}^k(-1)^{j+1}f(PU_j)|_{\frac{2N}{N+4}}+|\sum_{j=1}^k(-1)^{j+1}f(PU_j)-\sum_{j=1}^k(-1)^{j+1}f(U_j)|_{\frac{2N}{N+4}}\Big)
\nonumber\\
:=&C(W_1+W_2).
\end{align}

\text{Estimate of $W_1$}. Let $r>0$, $\beta:={\frac{2N}{N+4}}$, then
\begin{align*}
W_1^{\beta}=&\left(\int_{\Omega_\varepsilon\backslash B(\xi_0,r)}+\int_{\Omega_\varepsilon\bigcap B(\xi_0,r)}\right)|f(V)-\sum_{j=1}^k(-1)^{j+1}f(PU_j)|^\beta
:=W_{11}+W_{12}.
\end{align*}
Since the fact $|V|\leq |\sum_{j=1}^kU_j|\leq c\sum_{j=1}^k\mu_{j\varepsilon}=O(\varepsilon^{\frac{(N-4)\theta}{4k}})$ in $\Omega_\varepsilon\backslash B(\xi_0,r)$, we deduce that
\begin{align*}
W_{11}=\int_{\Omega_\varepsilon\backslash B(\xi_0,r)}|f(V)-\sum_{j=1}^k(-1)^{j+1}f(PU_j)|^\beta=O(\varepsilon^{\frac{(N-4)\theta}{4k}p\beta}).
\end{align*}
Next, we decompose the set
${{\Omega_\varepsilon}\cap B(\xi_0,r)}=B(\xi_0,r)\setminus B(\xi_0,\varepsilon)$ into the union of non-overlapping annuli, i.e.,
\begin{align}\label{e4.27}
B(\xi_0,r)\backslash B(\xi_0,\varepsilon)=\bigcup_{l=1}^k A_l,
\end{align}
where for all $l=1,\cdots,k$,
\begin{align}\label{e4.28}
A_l=B(\xi_0,\sqrt{\mu_{l\varepsilon}\mu_{(l-1)\varepsilon}})\backslash B(\xi_0,\sqrt{\mu_{l\varepsilon}\mu_{(l+1)\varepsilon}})
\end{align}
with $\mu_{0\varepsilon}\mu_{1\varepsilon}=r^2$ and $\mu_{k\varepsilon}\mu_{(k+1)\varepsilon}=\varepsilon^2$.
We write
$$W_{12}=\sum_{l=1}^k\int_{A_l}|f(V)-\sum_{j=1}^k(-1)^{j+1}f(PU_j)|^\beta,$$
where
\begin{align*}
\int_{A_l}|f(V)-\sum_{j=1}^k(-1)^{j+1}f(PU_j)|^\beta
\leq  c\left(\sum_{i\neq l}\int_{A_l}|U_l^{p-1}U_i|^{\beta}+\sum_{i\neq l}\int_{A_l}U_i^{p\beta}\right)\ \text{for } l=1,...,k.
\end{align*}
From Lemma \ref{lem6.6} and $\frac{N\theta}{2k}=\frac{(N-4)\theta}{2k}\frac{p\beta}{2}$, we get $\int_{A_l}U_i^{p\beta}=\int_{A_l}U_i^{p+1}=O(\varepsilon^{\frac{N\theta}{2k}})=O(\varepsilon^{\frac{(N-4)\theta}{2k}\frac{p\beta}{2}}).$\\
Combining with Lemma \ref{lem6.8}, we get
\begin{align}\label{e4.35}
|W_{1}|\leq
\begin{cases}c\varepsilon^{\frac{(N-4)\theta}{2k}\frac{p}{2}}, &\text{ if }N\geq 13,\\
  c\varepsilon^{\frac{(N-4)\theta}{2k}} |\ln \varepsilon|, &\text{ if }N=12,\\
  c\varepsilon^{\frac{(N-4)\theta}{2k}}, &\text{ if }5 \leq N \leq 11.\end{cases}
\end{align}

\text{Estimate of $W_2$}. From the mean value Theorem, we have that
\begin{align*}
\int_{\Omega_\varepsilon}|(PU_i)^p-U_i^p|^\beta \leq & c\left(\int_{\Omega_\varepsilon}|U_i^{p-1}(PU_i-U_i)|^\beta+\int_{\Omega_\varepsilon}|PU_i-U_i|^{p\beta}\right)
:=c(W_{21}+W_{22}).
\end{align*}
Using Lemma \ref{lem6.1}, we get
\begin{align*}
|W_{21}|\leq c\int_{\Omega_\varepsilon}\bigg|U_i^{p-1}\bigg(\mu_{i\varepsilon}^{\frac{N-4}{2}}+\frac{\varepsilon^{N-2}}{\mu_{i\varepsilon}^{\frac{N}{2}}}\frac{1}{|x-\xi_0|^{N-4}}+\frac{\varepsilon^{N-2}}{\mu_{i\varepsilon}^{\frac{N-4}{2}}}\frac{1}{|x-\xi_0|^{N-2}}
+\frac{\varepsilon^{N}}{\mu_{i\varepsilon}^{\frac{N}{2}}}\frac{1}{|x-\xi_0|^{N-2}}\bigg)\bigg|^\beta.
\end{align*}
Since
\begin{align*}
\int_{\Omega_\varepsilon}U_i^{(p-1)\beta}=&\int_{\Omega_\varepsilon}\left(\frac{\mu_{j\varepsilon}^4}{(\mu_{j\varepsilon}^2+|x-\xi_{j\varepsilon}|^2)^4}\right)^\beta
=\begin{cases}O(\mu_{j\varepsilon}^{4\beta}), &\text{ if }N\geq 13,\\
  O(\mu_{j\varepsilon}^{4\beta}|\ln \mu_{j\varepsilon}|^\beta), &\text{ if }N=12,\\
  O(\mu_{j\varepsilon}^{N-4\beta}), &\text{ if }5\leq N\leq 11.\end{cases}
\end{align*}
Then
\begin{align}\label{e4.36}
|W_{2}|\leq
\begin{cases}c\varepsilon^{\frac{(N-4)\theta}{2k}\frac{p}{2}}, &\text{ if }N\geq 13,\\
  c\varepsilon^{\frac{(N-4)\theta}{2k}} |\ln \varepsilon|, &\text{ if }N=12,\\
  c\varepsilon^{\frac{(N-4)\theta}{2k}}, &\text{ if }5 \leq N \leq 11.\end{cases}
\end{align}
The result follows  from (\ref{rdfg}), (\ref{e4.35}) and (\ref{e4.36}).
\end{proof}

\noindent{\it Proof of Proposition \ref{pro3.1}}:
By Proposition \ref{pro3.3}, we define $T(\phi):= L^{-1}(N(\phi)+R)\text{ for } \phi\in K^\perp$.
Next, we claim that $T$ is a contraction map. Set
\begin{align*}
q(\varepsilon):=c\begin{cases}\varepsilon^{ \frac{(N-4)\theta}{2k}\frac{p}{2}}, &\text{ if }N\geq 13,\\
  \varepsilon^{ \frac{(N-4)\theta}{2k}}|\ln \varepsilon|, &\text{ if }N=12,\\
  \varepsilon^{ \frac{(N-4)\theta}{2k}}, &\text{ if }5\leq N\leq 11.\end{cases}
\end{align*}
We first prove that there exists a properly subset $\Lambda:=\{\phi: \|\phi\|\leq q(\varepsilon)\}$ of $H^2\cap H^1_0(\Omega_\varepsilon)$ such that $T:\Lambda\rightarrow \Lambda$.
From Lemma \ref{lem3.2}, we have $\|T(\phi)\|\leq c\|N(\phi)+R\|\leq c(\|N(\phi)\|+\|R\|)$.

By (\ref{e2.1}), we get $\|N(\phi)\|\leq c|f(V+\phi)-f(V)-f'(V)\phi|_{\frac{2N}{N+4}}$.
From Lemma A.1 in \cite{Musso}, we have
\begin{align*}
|f(V+\phi)-f(V)-f'(V)\phi|\leq
\begin{cases}c|\phi|^p, &\text{ if }N>12,\\
  c(|V|^{p-2}|\phi|^2+|\phi|^p), &\text{ if }5 \leq N\leq 12.\end{cases}
\end{align*}
Thus $\|N(\phi)\|\leq c|\phi|_{\frac{2N}{N+4}}^{\min\{2,p\}}\leq c\|\phi\|^{\min\{2,p\}}$.
Then by Lemma \ref{lem4.4}, we obtain $\|T(\phi)\|\leq q(\varepsilon)$.
\par Secondly, we prove that $T$ is a contraction map.
Similar to the proof above, we get that, for some $0<t<1$,
\begin{align*}
\|T(\phi_1)-T(\phi_2)\|\leq &c\|N(\phi_1)-N(\phi_2)\|
\nonumber\\
\leq &c|f(V+\phi_1)-f(V+\phi_2)-f'(V)(\phi_1-\phi_2)|_{\frac{2N}{N+4}}
\leq t\|\phi_1-\phi_2\|.
\end{align*}
Thus $T:\Lambda\rightarrow \Lambda$ is a contraction map, then it has a unique fixed point $\phi \in \Lambda$.
\par Finally, we show that $(\mu,\sigma)\mapsto \phi(\mu,\sigma)$ is a $C^1$-map with respect to $\mu, \sigma$.
Let us define $F(\mu,\sigma,\phi)=L(\phi)-N(\phi)-R$, then $F(\mu,\sigma,\phi)=0$ and $D_\phi F(\mu,\sigma,\phi)=L(\zeta)-D_\phi N(\phi)[\zeta]$ for all $\zeta\in K^{\perp}$. We choose $\zeta\in K^{\perp}$ such that $D_\phi F(\mu,\sigma,\phi)=L(\zeta)-D_\phi N(\phi)[\zeta]=0$.

From Lemma \ref{lem3.2}, we have
\begin{align}\label{e3.37}
\|\zeta\|\leq c\|L(\zeta)\|.
\end{align}
However, by (\ref{e2.1}), the H\"{o}lder inequality and $|f'(V+\phi)-f'(V)|_{\frac{N}{4}}=o(1)$, we see that
\begin{align}\label{e3.38}
\|D_\phi N(\phi)[\zeta]\|=&\parallel \Pi^{\perp}\{i^*[f'(V+\phi)\zeta-f'(V)\zeta]\}\parallel
\nonumber\\
\leq &c|f'(V+\phi)\zeta-f'(V)\zeta|_{\frac{2N}{N+4}}\leq c|f'(V+\phi)-f'(V)|_{\frac{N}{4}}|\zeta|_{\frac{2N}{N-4}}\leq co(1)\|\zeta\|.
\end{align}
By (\ref{e3.37}) and (\ref{e3.38}), we have $\zeta=0$, then $D_\phi F(\mu,\sigma,\phi)$ is injective for $\varepsilon$ small enough. Using the implicit function Theorem, we prove $(\mu,\sigma)\mapsto \phi(\mu,\sigma)$ is a $C^1$-map with respect to $\mu$, $\sigma$.
\qed

\section{Expansion of the energy functional}\label{exp}

This section is devoted to the proof of Proposition \ref{pro3.4}.

\noindent {\it Proof of Proposition \ref{pro3.4} (i)}:
By (\ref{e2.9}), there exist constants $c_i^l$, $l=0,\cdots,N$, $i=1,\cdots,k$, such that
\begin{align}\label{e3.41}
\nabla I(\mu,\sigma)=\langle V+\phi-i^*(f(V+\phi)), \nabla V+\nabla \phi \rangle
=\sum_{l=0}^N\sum_{i=1}^kc_i^l\langle PZ_i^l, \nabla V+\nabla \phi \rangle.
\end{align}
If we compute (\ref{e3.41}) at $(\mu,\sigma)$, which is a critical point of $I$, we then get
\begin{align}\label{e3.42}
\sum_{l=0}^N\sum_{i=1}^kc_i^l\langle PZ_i^l, \nabla V+\nabla \phi \rangle=0.
\end{align}
Since
\begin{align}\label{e3.43}
\partial_{\mu_h}V=\varepsilon^{\frac{2h-1}{2k}\theta}PZ_h^0+\varepsilon^{\frac{2h-1}{2k}\theta}\sum_{j=1}^NPZ_h^j\sigma_{hj}, \qquad
  \partial_{\sigma_r^j}V=\mu_{r\varepsilon}PZ_r^j,
\end{align}
and by Lemma \ref{lem6.3}, we have, for $h=1,\cdots,k$
\begin{align}\label{e3.44}
\sum_{l=0}^N\sum_{i=1}^kc_i^l\langle PZ_i^l, \partial_{\mu_h}V \rangle=&\varepsilon^{\frac{2h-1}{2k}\theta}\sum_{l=0}^N\sum_{i=1}^kc_i^l\langle PZ_i^l, PZ_h^0 \rangle+\varepsilon^{\frac{2h-1}{2k}\theta}\sum_{l=0}^N\sum_{i=1}^k\sum_{j=1}^N\sigma_{hj}c_i^l\langle PZ_i^l, PZ_h^j \rangle
\nonumber\\
=&\frac{\varepsilon^{\frac{2h-1}{2k}\theta}}{\mu_{h\varepsilon}^2}c_h^0[c_0+o(1)]+\frac{\varepsilon^{\frac{2h-1}{2k}\theta}}{\mu_{h\varepsilon}^2}\sum_{j=1}^N\sigma_{hj}c_h^j[c_1+o(1)], \end{align}
and for $r=1,\cdots,k$, $j=1,\cdots,N$,
\begin{align}\label{e3.45}
\sum_{l=0}^N\sum_{i=1}^kc_i^l\langle PZ_i^l, \partial_{\sigma_r^j}V \rangle=\mu_{r\varepsilon}\sum_{l=0}^N\sum_{i=1}^kc_i^l\langle PZ_i^l, PZ_r^j \rangle
=\frac{1}{\mu_{r\varepsilon}}c_r^j[c_1+o(1)].
\end{align}
Moreover, since $\langle PZ_i^l,\phi \rangle=0$ for $\phi\in K^\perp$ and by Proposition \ref{pro3.1}, we have
\begin{align}\label{e3.46}
\langle PZ_i^l,\partial_s\phi \rangle=-\langle \partial_sPZ_i^l,\phi \rangle=O(\|\partial_sPZ_i^l\|\|\phi\|)=o(\|\partial_sPZ_i^l\|),
\end{align}
where $\partial_s$ denotes one of the components of the gradient of $\phi$ or $PZ_i^l$.
\par Since $\Delta^2 PU=U^p$ in $\Omega_\varepsilon$, $|\partial_{\mu_{i\varepsilon}}Z^l_i|\leq c\frac{U}{{\mu_{i\varepsilon}}^2}$ , $|Z^l_i|\leq c\frac{U}{{\mu_{i\varepsilon}}}$ and $||U^{p}||_{\frac{2N}{N+4}}=O(1)$, we get
\begin{align*}
\|\partial_{\mu_i}PZ_i^l\|\leq c\varepsilon^{\frac{2i-1}{2k}\theta}\parallel U^{p-1}\partial_{\mu_{i\varepsilon}}Z^l_i+U^{p-2}{Z^l_i}^2\parallel_{\frac{2N}{N+4}}
\leq c\frac{\varepsilon^{\frac{2i-1}{2k}\theta}}{{\mu_{i\varepsilon}}^2}||U^{p}||_{\frac{2N}{N+4}}
\leq c\frac{\varepsilon^{\frac{2i-1}{2k}\theta}}{{\mu_{i\varepsilon}}^2}.
\end{align*}
If $i\neq h$, we have $\partial_{\mu_h}PZ_i^l=0$ and $\|\partial_{\mu_h}PZ_i^l\|=0$.
We can estimate $\|\partial_{\sigma_r^j}PZ_i^l\|$ in the same way.
Then
\begin{align}\label{e3.47}
\|\partial_sPZ_i^l\|=
\begin{cases}\|\partial_{\mu_h}PZ_i^l\|=0, &\text{ if }h\neq i, h=1,\cdots,k,\\
  \|\partial_{\mu_h}PZ_h^l\|=O(\frac{\varepsilon^{\frac{2h-1}{2k}\theta}}{\mu_{h\varepsilon}^2}), &\text{ if }h=i,\\
  \|\partial_{\sigma_r^j}PZ_i^l\|=0, &\text{ if }r\neq i,\\
  \|\partial_{\sigma_r^j}PZ_r^l\|=O(\frac{1}{\mu_{r\varepsilon}}), &\text{ if }r=i.\end{cases}
\end{align}
Thus, we deduce that $c_i^l=0$ for all $l=0,\cdots,N$, $i=1,\cdots,k$.
\qed

\smallskip

We next give the $C^1$-expansion of functional $I(\mu,\sigma)$.

\begin{lemma}\label{lem4.2}
For any $d>0$ small but fixed, there exist $\varepsilon_0>0$ and $c>0$ such that   for any $\varepsilon \in(0,\varepsilon_0)$,
\begin{align}\label{e4.3}
I(\mu,\sigma):=J_\varepsilon(V+\phi)=J_\varepsilon(V)+o(\varepsilon^{\frac{(N-4)\theta}{2k}})
\end{align}
$C^1$-uniformly with respect to $\mu$ and $\sigma$ satisfying (\ref{e2.8}).
\end{lemma}
\begin{proof}
We will show that
\begin{align}\label{c0}
J_\varepsilon(V+\phi)-J_\varepsilon(V)=o(\varepsilon^{\frac{(N-4)\theta}{2k}})
\end{align}
and
\begin{align}\label{c1}
\nabla_{\mu, \sigma}[J_\varepsilon(V+\phi)-J_\varepsilon(V)]=o(\varepsilon^{\frac{(N-4)\theta}{2k}}).
\end{align}
Indeed, we have
\begin{align}\label{e4.4}
J_\varepsilon(V+\phi)-J_\varepsilon(V)
=&\frac{1}{2}\|\phi\|^2-\int_{\Omega_\varepsilon}\big(f(V)-\sum_{j-1}^k(-1)^{j+1}f(PU_j)\big)\phi
\nonumber\\
&-\int_{\Omega_\varepsilon}\big(F(V+\phi)-F(V)-f(V)\phi\big).
\end{align}
We note that $\|\phi\|^2=o(\varepsilon^{\frac{(N-4)\theta}{2k}})$ by Proposition \ref{pro3.1}. By H\"{o}lder inequality and Lemma \ref{lem4.4}, we have
\begin{align}\label{e4.5}
\left|\int_{\Omega_\varepsilon}[f(V)-\sum_{j=1}^k(-1)^{j+1}f(PU_j)]\phi\right|\leq & c|f(V)-\sum_{j=1}^k(-1)^{j+1}f(PU_j)|_{\frac{2N}{N+4}}|\phi|_{\frac{2N}{N-4}}
=o(\varepsilon^{\frac{(N-4)\theta}{2k}}).
\end{align}
Moreover, by the mean value Theorem and H\"{o}lder inequality,  for some $t\in (0,1)$,
\begin{align}\label{e4.6}
&\left|\int_{\Omega_\varepsilon}(F(V+\phi)-F(V)-f(V)\phi)\right|
\leq  c\int_{\Omega_\varepsilon}|f'(V+t\phi)\phi^2|
\nonumber\\
\leq &c\int_{\Omega_\varepsilon}|V|^{p-1}\phi^2+\int_{\Omega_\varepsilon}|\phi|^{p+1}
\leq c\left||V|^{p-1}\right|_{\frac{N}{4}}|\phi|^2_{\frac{2N}{N-4}}+c|\phi|^{p+1}_{\frac{2N}{N-4}}
=o(\varepsilon^{\frac{(N-4)\theta}{2k}}).
\end{align}
Thus (\ref{c0}) holds.

We next prove (\ref{c1}).
We have
\begin{align}\label{e4.7}
\nabla J_\varepsilon(V+\phi)-\nabla J_\varepsilon(V)
=[J'(V+\phi)-J'(V)][\nabla V]+J'(V+\phi)[\nabla \phi].
\end{align}
By (\ref{e3.43}), $\partial_sV$ is a linear combination of $\varepsilon^{\frac{(2j-1)\theta}{2k}}PZ^h_j$ with coefficients uniformly bounded as $\varepsilon \rightarrow 0$ for any  $\mu$, $\sigma$ satisfying (\ref{e2.8}).  Thus, fix $j$,
\begin{align}\label{e4.8}
&[J'(V+\phi)-J'(V)][\varepsilon^{\frac{(2j-1)\theta}{2k}}PZ^h_j]
\nonumber\\
=&-\int_{\Omega_\varepsilon}f'(V)\phi\varepsilon^{\frac{(2j-1)\theta}{2k}}[PZ^h_j-Z^h_j]-\int_{\Omega_\varepsilon}[f'(V)-f'(U_j)]\phi\varepsilon^{\frac{(2j-1)\theta}{2k}}Z^h_j
\nonumber\\
&-\int_{\Omega_\varepsilon}[f(V+\phi)-f(V)-f'(V)\phi]\varepsilon^{\frac{(2j-1)\theta}{2k}}PZ^h_j
:=H_1+H_2+H_3.
\end{align}
By H\"{o}lder inequality and Proposition \ref{pro3.1}, it follows that
\begin{align}\label{e4.9}
|H_1|\leq c|f'(V)|_{\frac{N}{4}}|\phi|_{\frac{2N}{N-4}}|PZ^h_j-Z^h_j|_{\frac{2N}{N-4}}\varepsilon^{\frac{(2j-1)\theta}{2k}}=o(\varepsilon^{\frac{(N-4)\theta}{2k}}).
\end{align}
Since $|\varepsilon^{\frac{(2j-1)\theta}{2k}}PZ^h_j|\leq |\varepsilon^{\frac{(2j-1)\theta}{2k}}Z^h_j|=c\frac{\varepsilon^{\frac{(2j-1)\theta}{2k}}\mu_{j\varepsilon}^{\frac{N-4}{2}}|x^h-\xi_j^h|}{(\mu_{j\varepsilon}^2+|x-\xi_j|^2)^{\frac{N-2}{2}}}
=cU_j\frac{|x^h-\xi_j^h|}{\mu_{j\varepsilon}^2+|x-\xi_j|^2}
\leq cU_j$,
we get
\begin{align}\label{e4.10}
|H_2|\leq  c\int_{\Omega_\varepsilon}|V^{p-1}-U_j^{p-1}||\phi|U_j
=c\left(\int_{\Omega_\varepsilon\setminus B(\xi_0,r)}+\int_{A_j}+\sum_{i\neq j,i=1}^k\int_{A_i}\right)|V^{p-1}-U_j^{p-1}||\phi|U_j,
\end{align}
where
$$|\int_{\Omega_\varepsilon\setminus B(\xi_0,r)}|V^{p-1}-U_j^{p-1}||\phi|U_j|\leq c\left||V^{p-1}-U_j^{p-1}|U_j\right|_{\frac{2N}{N+4}}|\phi|_{\frac{2N}{N-4}}=o(\varepsilon^{\frac{(N-4)\theta}{2k}}).
$$

If $N\geq 13$, we have
\begin{align}\label{e4.11}
&\int_{A_j}|V^{p-1}-U_j^{p-1}||\phi|U_j\leq  c\int_{A_j}U_j^{p-1}|(PU_j-U_j)+\sum_{i\neq j,i=1}^kPU_i||\phi|
\nonumber\\
\leq &c|U_j^{p-1}|_{\frac{N}{4}}|PU_j-U_j|_{\frac{2N}{N-4}}|\phi|_{\frac{2N}{N-4}}
+c\sum_{i\neq j,i=1}^k|U_j^{p-1}|_{\frac{N}{4}}|U_i|_{\frac{2N}{N-4}}|\phi|_{\frac{2N}{N-4}}
=o(\varepsilon^{\frac{(N-4)\theta}{2k}}).
\end{align}
The case $ 5\leq N\leq 12$ can be treated similarly. Moreover
$$\sum_{i\neq j,i=1}^k\int_{A_i}|V^{p-1}-U_j^{p-1}||\phi|U_j=o(\varepsilon^{\frac{(N-4)\theta}{2k}}).$$
Thus, we have $|H_2|\leq o(\varepsilon^{\frac{(N-4)\theta}{2k}})$.

Next we estimate $H_3$. By the mean value Theorem, for some $t\in [0,1]$,
\begin{align}\label{e4.12}
|H_3|\leq &\int_{\Omega_\varepsilon}|f''(V+t\phi)\phi^2|
\leq c\int_{\Omega_\varepsilon}|V|^{p-1}\phi^2+c\int_{\Omega_\varepsilon}|\phi|^{p+1}
\nonumber\\
\leq &c\left||V|^{p-1}\right|_{\frac{N}{4}}|\phi|_{\frac{2N}{N-4}}^2+c|\phi|_{\frac{2N}{N-4}}^{p+1}
=o(\varepsilon^{\frac{(N-4)\theta}{2k}}).
\end{align}
Then $[J'(V+\phi)-J'(V)][\varepsilon^{\frac{(2j-1)\theta}{2k}}PZ^h_j]=o(\varepsilon^{\frac{(N-4)\theta}{2k}})$, namely, $[J'(V+\phi)-J'(V)][\nabla V]=o(\varepsilon^{\frac{(N-4)\theta}{2k}})$.

Finally, by (\ref{e2.9})
\begin{align}\label{e4.13}
J'(V+\phi)[\nabla \phi]=\sum_{l=0}^N\sum_{i=1}^kc_i^l\langle PZ^l_i,\nabla \phi \rangle.
\end{align}
Since $|\langle PZ^l_i,\nabla \phi \rangle|\leq c|Z^l_i|_{\frac{2N}{N-4}}|\phi|_{\frac{2N}{N-4}}\leq c\mu_{i\varepsilon}^{-1}|\phi|_{\frac{2N}{N-4}}$ and $|c_i^l|\leq c\mu_{i\varepsilon}\|\phi\|$, we deduce
\begin{align}\label{e4.15}
|c_i^l\langle PZ^l_i,\nabla \phi \rangle|\leq c|\phi|_{\frac{2N}{N-4}}\|\phi\|\leq c\|\phi\|^2.
\end{align}
Thus $|J'(V+\phi)[\nabla \phi]|=O(\|\phi\|^2)=o(\varepsilon^{\frac{(N-4)\theta}{2k}}).$ Then we finish the proof.
\end{proof}

\begin{lemma}\label{lem6.5}
It holds
\begin{align}\label{J0}
\sum_{j=1}^k J_{\varepsilon}(PU_{j})=&\frac{2kc_1}{N}\alpha_N^{p+1}+\frac{\alpha_N^{p+1}}{2}\left(c_2H(\xi_0,\xi_0)\mu_1^{N-4}+\frac{c_3\Delta U_{1,0}(\sigma_k)U_{1,0}(\sigma_k)}{\mu_k^{N-2}}\right)\varepsilon^{\frac{N-4}{2k}\theta}
+o(\varepsilon^{\frac{N-4}{2k}\theta}),
\end{align}
where $c_1$, $c_2$ and $c_3$ are given in (\ref{e4.17}).
\end{lemma}
\begin{proof}
Let $U_{j}:=U_{\mu_{j\varepsilon},\xi_{j\varepsilon}}$, $PU_{j}:=P_{\varepsilon}U_{\mu_{j\varepsilon},\xi_{j\varepsilon}}$. Since $PU_{j}$ satisfies (\ref{PU}), we get, for some $t\in [0,1]$,
\begin{align}\label{e6.6}
J_{\varepsilon}(PU_{j})=&\frac{1}{2}\int_{\Omega_\varepsilon}U_{j}^p PU_{j}-\frac{1}{p+1}\int_{\Omega_\varepsilon}PU_{j}^{p+1}
\nonumber\\
=&\frac{2}{N}\int_{\Omega_\varepsilon}U_{j}^{p+1}-\frac{1}{2}\int_{\Omega_\varepsilon}U_{j}^p (PU_{j}-U_{j})-\frac{p}{2}\int_{\Omega_\varepsilon}[U_{j}+t(PU_{j}-U_{j})]^{p-1}[PU_{j}-U_{j}]^2
\nonumber\\
:=&I_1+I_2+I_3.
\end{align}
From \cite[Proposition 3.3]{Ala}, we have
\begin{align}\label{I1}
I_1
=\frac{2}{N}\alpha_N^{p+1}\int_{\mathbb{R}^N}\left(\frac{1}{1+|z|^2}\right)^N+O\left(\varepsilon^{\frac{2j-1}{2k} \theta N}\right),
\end{align}
\begin{align}\label{I2}
I_2=&\frac{\alpha_N^{p+1}}{2}\int_{\mathbb{R}^N}\frac{1}{(1+|z|^2)^{\frac{N+4}{2}}}H(\xi_0,\xi_0){\mu_{j}^{N-4}}\varepsilon^{\frac{2j-1}{2k}\theta(N-4)}(1+o(1))
\nonumber\\
&+\left(\frac{\varepsilon}{\mu_{j\varepsilon}}\right)^{N-2}\left(\frac{-3(N-2)}{4}meas(\mathbb{S}^{N-1})\Delta U_{1,0}(\sigma_{j})U_{1,0}(\sigma_{j})+o(1)\right),
\end{align}
and
\begin{align}\label{I3}
|I_3|=O\left(\mu_{j\varepsilon}^{N}+\frac{\varepsilon^{N}}{\mu_{j\varepsilon}^{N}}\right)=o\left(\varepsilon^{\frac{2j-1}{2k}\theta(N-4)}\right).
\end{align}
From (\ref{e6.6})-(\ref{I3}), we get (\ref{J0}).

Next, we give the $C^1$ expansion of $\sum_{j=1}^k J_{\varepsilon}(PU_{j})$.
Let $\partial_s$ denote $\partial_{\mu_j}$ for $j= 1,\cdots,k$ and $\partial_{\sigma^i_r}$ for $r=1,\cdots,k-1$ and $i=1,\cdots,N$.
By a standard argument, we have
\begin{align}\label{eq6.51}
\partial_s J_{\varepsilon}(PU_{j})
=&\partial_s\left(-\int_{\Omega_\varepsilon}|U_{j}|^{p}(PU_{j}-U_{j})\right)+\int_{\Omega_\varepsilon}|U_{j}|^{p}\partial_s(PU_{j}-U_{j})
\nonumber\\
&+O\left(\int_{\Omega_\varepsilon}|U_{j}|^{p-2}U_{j}(PU_{j}-U_{j})^2\right).
\end{align}
Using Lemma \ref{lem6.7}, then for $\partial_s=\partial_{\mu_j}$, $j=1,\cdots,k$, we have
\begin{align}\label{eq6.52}
\partial_s\left(-\int_{\Omega_\varepsilon}|U_{j}|^{p}(PU_{j}-U_{j})\right)
=&\partial_{\mu_{j}}\bigg(\alpha_N^{p+1}\int_{\Omega_\varepsilon}\frac{1}{(1+|z|^2)^{\frac{N+4}{2}}}H(\xi_0,\xi_0)\mu_{j\varepsilon}^{N-4}(1+o(1))
\nonumber\\
-&\frac{3(N-2)}{2}{meas}(\mathbb{S}^{N-1})\Delta U_{1,0}(\sigma)U_{1,0}(\sigma)(\frac{\varepsilon}{\mu_{j\varepsilon}})^{N-2}(1+o(1))\bigg)
\nonumber\\
=&\frac{(N-4)\alpha_N^{p+1}}{\mu_j}\int_{\Omega_\varepsilon}\frac{1}{(1+|z|^2)^{\frac{N+4}{2}}}H(\xi_0,\xi_0)\mu_{j\varepsilon}^{N-4}(1+o(1))
\nonumber\\
+&\frac{3(N-2)^2}{2\mu_j}{meas}(\mathbb{S}^{N-1})\Delta U_{1,0}(\sigma)U_{1,0}(\sigma)(\frac{\varepsilon}{\mu_{j\varepsilon}})^{N-2}(1+o(1)).
\end{align}
By Lemma \ref{lem6.1}, we have
\begin{align*}
&\int_{\Omega_\varepsilon}|U_{j}|^{p}\partial_s(PU_{j}-U_{j})\\
=&\int_{\Omega_\varepsilon}|U_{j}|^{p}\partial_{\mu_{j}}\bigg(-\alpha_N\mu_{j\varepsilon}^{\frac{N-4}{2}}H(x,\xi_{j\varepsilon})-a_1 \varphi_1(\frac{x-\xi_0}{\varepsilon})-a_2\varphi_2(\frac{x-\xi_0}{\varepsilon})
+R_{\varepsilon}(x)\bigg).
\end{align*}
By a direct computation, we find
\begin{align*}
\int_{\Omega_\varepsilon}|U_{j}|^{p}\partial_{\mu_{j}}\left(-\alpha_N\mu_{j\varepsilon}^{\frac{N-4}{2}}H(x,\xi_{j\varepsilon})\right)
=-\frac{(N-4)\alpha_N^{p+1}}{2\mu_j}\int_{\Omega_\varepsilon}\frac{1}{(1+|z|^2)^{\frac{N+4}{2}}}H(\xi_0,\xi_0)\mu_{j\varepsilon}^{N-4}(1+o(1)).
\end{align*}
Let $x-\xi_j=\mu_{j\varepsilon}z$, then
\begin{align*}
&\int_{\Omega_\varepsilon}|U_{j}|^{p}\partial_{\mu_{j}}\left(-a_1 \varphi_1(\frac{x-\xi_0}{\varepsilon})\right)
\nonumber\\
=&\int_{\bar{{\Omega}}_\varepsilon}\left(\frac{\mu_{j\varepsilon}}{\mu_{j\varepsilon}^2+|\mu_{j\varepsilon}z|^2}\right)^{\frac{N+4}{2}}\mu_{j\varepsilon}^N\partial_{\mu_{j}}\left(\frac{\Delta U_{1,0}(\sigma)}{2(N-4)}\frac{\varepsilon^{N-2}}{\mu_{j\varepsilon}^{\frac{N}{2}}}\frac{1}{|\mu_{j\varepsilon}(z+\sigma_j)|^{N-4}}\right)
\nonumber\\
=&\frac{3N-8}{2\mu_{j}}\left(\frac{\varepsilon}{\mu_{j\varepsilon}}\right)^{N-2}
\left(-\frac{N-2}{2}{meas}(\mathbb{S}^{N-1})\Delta U_{1,0}(\sigma)U_{1,0}(\sigma)+o(1)\right),
\end{align*}
and
\begin{align*}
&\int_{\Omega_\varepsilon}|U_{j}|^{p}\partial_{\mu_{j}}\left(-a_2\varphi_2(\frac{x-\xi_0}{\varepsilon})\right)
\nonumber\\
=&\int_{\tilde{\Omega}_\varepsilon}\frac{\mu_{j\varepsilon}^{\frac{N-4}{2}}}{(1+|z|^2)^{\frac{N+4}{2}}}\partial_{\mu_{j}}\left(-\frac{\Delta U_{1,0}(\sigma)}{2(N-4)}\frac{\varepsilon^{N}}{\mu_{j\varepsilon}^{\frac{3N-4}{2}}}\frac{1}{|z+\sigma_j|^{N-2}}-U_{1,0}(\sigma)\frac{\varepsilon^{N-2}}{\mu_{j\varepsilon}^{\frac{3N-8}{2}}}\frac{1}{|z+\sigma_j|^{N-2}}\right)
\nonumber\\
=&\frac{3N-8}{2\mu_{j}}\left(\frac{\varepsilon}{\mu_{j\varepsilon}}\right)^{N-2}\left(-(N-2){meas}(\mathbb{S}^{N-1})\Delta U_{1,0}(\sigma)U_{1,0}(\sigma)+o(1)\right).
\end{align*}
From Lemma \ref{lem6.1}, we have that $\left|\int_{\Omega_\varepsilon}|U_{j}|^{p}\partial_{\mu_{j}}R_{\varepsilon}(x)\right|\leq C\frac{1}{\mu_{j\varepsilon}}\left|\int_{\Omega_\varepsilon}|U_{j}|^{p}R_{\varepsilon}(x)\right|=O\left(\left(\frac{\varepsilon}{\mu_{j\varepsilon}}\right)^{N-2}\right).$
Then
\begin{align}\label{eq6.53}
\int_{\Omega_\varepsilon}|U_{j}|^{p}\partial_{\mu_{j}}(PU_{j}-&U_{j})
=\frac{-(N-4)\alpha_N^{p+1}}{2\mu_j}\int_{\Omega_\varepsilon}\frac{1}{(1+|z|^2)^{\frac{N+4}{2}}}H(\xi_0,\xi_0)\mu_{j\varepsilon}^{N-4}(1+o(1))
\nonumber\\
&-\frac{3(N-4)}{4\mu_{j}}\left(\frac{\varepsilon}{\mu_{j\varepsilon}}\right)^{N-2}\left((N-2){meas}(\mathbb{S}^{N-1})\Delta U_{1,0}(\sigma)U_{1,0}(\sigma)+o(1)\right).
\end{align}
By Lemma \ref{lem6.2}, we get $O\left(\int_{\Omega_\varepsilon}p|U_{j}|^{p-2}U_{j}(PU_{j}-U_{j})^2\right)=O\left(\mu_{j\varepsilon}^N+\left(\frac{\varepsilon}{\mu_{j\varepsilon}}\right)^{N}\right),$
then combining this estimate with (\ref{eq6.51}), (\ref{eq6.52}) and (\ref{eq6.53}), we obtain
\begin{align*}
\partial_{\mu_j} J_{\varepsilon}(PU_{j})=&\frac{(N-4)\alpha_N^{p+1}}{2\mu_j}\int_{\Omega_\varepsilon}\frac{1}{(1+|z|^2)^{\frac{N+4}{2}}}H(\xi_0,\xi_0)\mu_{j\varepsilon}^{N-4}(1+o(1))
\nonumber\\
&+\frac{3(N-2)(N-4)}{4\mu_j}{meas}(\mathbb{S}^{N-1})\Delta U_{1,0}(\sigma_j)U_{1,0}(\sigma_j)(\frac{\varepsilon}{\mu_{j\varepsilon}})^{N-2}(1+o(1))
\nonumber\\
= &\partial_{\mu_j}\left(\frac{\alpha_N^{p+1}}{2}\left(c_2H(\xi_0,\xi_0)\mu_{j\varepsilon}^{N-4}+c_3\frac{\Delta U_{1,0}(\sigma_j)U_{1,0}(\sigma_j)\varepsilon^{N-2}}{\mu_{j\varepsilon}^{N-2}}\right)\right)+o(\varepsilon^{\frac{N-4}{2k}\theta}).
\end{align*}
As the same way, for $j=1,\cdots,N$ and $r=1,\cdots,k-1$, we have
\begin{align*}
\partial_{\sigma_r^j}\left(J_\varepsilon(PU_{j})\right)=\partial_{\sigma_r^j}\left(\frac{\alpha_N^{p+1}}{2}\left(c_2H(\xi_0,\xi_0)\mu_{j\varepsilon}^{N-4}+c_3\frac{\Delta U_{1,0}(\sigma_j)U_{1,0}(\sigma_j)\varepsilon^{N-2}}{\mu_{j\varepsilon}^{N-2}}\right)\right)+o(\varepsilon^{\frac{N-4}{2k}\theta}).
\end{align*}
\end{proof}

\begin{lemma}\label{thm4.3}
For $d>0$ small but fixed, there exist $\varepsilon_0>0$ and $c>0$ such that  for any $\varepsilon \in (0, \varepsilon_0)$,
\begin{align}\label{e4.16}
J_\varepsilon (V)=&\frac{2kc_1}{N}\alpha_N^{p+1}+\frac{\alpha_N^{p+1}}{2}\Bigg\{c_2H(\xi_0,\xi_0)\mu_1^{N-4}+c_3\frac{\Delta U_{1,0}(\sigma_k)U_{1,0}(\sigma_k)}{\mu_k^{N-2}}
\nonumber\\
&+2\sum_{l=1}^{k-1}\Gamma(\sigma_l)(\frac{\mu_{l+1}}{\mu_{l}})^{\frac{N-4}{2}}\Bigg\}\varepsilon^{\frac{N-4}{2k}\theta}+o(\varepsilon^{\frac{N-4}{2k}\theta})
\end{align}
$C^1$-uniformly with respect to $\mu_j$, $\sigma_j$ satisfying (\ref{e2.8}). Here
$c_1$, $c_2$ and function $\Gamma$ are defined in Proposition \ref{pro3.4}.
\end{lemma}
\begin{proof}
We write $U_{\mu_{j\varepsilon},\xi_{j\varepsilon}}$ instead of $U_{j}$, $P_{\varepsilon}U_{\mu_{j\varepsilon},\xi_{j\varepsilon}}$ instead of $PU_{j}$.
Then
\begin{align}\label{e4.21}
J_{\varepsilon}(V)=&\frac{1}{2}\sum_{j=1}^k \int_{\Omega_\varepsilon}|\Delta PU_{j}|^2+\int_{\Omega_\varepsilon}\sum_{i>j}(-1)^{i+j} \Delta PU_{i}\Delta PU_{j}-\frac{1}{p+1}\int_{\Omega_\varepsilon}|\sum_{j=1}^k(-1)^{j+1}PU_{j}|^{p+1}
\nonumber\\
=&\sum_{j=1}^kJ_{\varepsilon}(PU_{j})-\frac{1}{p+1}\int_{\Omega_\varepsilon}\bigg(|\sum_{j=1}^k(-1)^{j+1}PU_{j}|^{p+1}-(p+1)\sum_{i>j}(-1)^{i+j} PU_{i}^p PU_{j}
\nonumber\\
&-\sum_{j=1}^k|PU_{j}|^{p+1}\bigg)
:=\sum_{j=1}^kJ_{\varepsilon}(PU_{j})-\frac{1}{p+1}\int_{\Omega_\varepsilon}M(x):=\sum_{j=1}^kJ_{\varepsilon}(PU_{j})+J_R.
\end{align}
Assume $r>0$, thus
\begin{align*}
J_R=-\frac{1}{p+1}\left(\int_{{\Omega_\varepsilon}\setminus B(\xi_0,r)}M(x)+\int_{{\Omega_\varepsilon}\cap B(\xi_0,r)}M(x)\right).
\end{align*}
It holds
\begin{align}\label{e4.26}
\left|\int_{{\Omega_\varepsilon}\setminus B(\xi_0,r)}M(x)\right|&\leq c\left(\sum_{j=1}^k\int_{{\Omega_\varepsilon}\setminus B(\xi_0,r)}U_j^{p+1}+\sum_{i\neq j}\int_{{\Omega_\varepsilon}\setminus B(\xi_0,r)}U_i^{p}U_j\right)
\nonumber\\
&\leq c\left(\sum_{j=1}^k{\mu_{j\varepsilon}}^{N}+\sum_{i\neq j}{\mu_{i\varepsilon}}^{\frac{N+4}{2}}{\mu_{j\varepsilon}}^{\frac{N-4}{2}}\right)
=O\left(\varepsilon^{\frac{N\theta}{2k}}\right).
\end{align}
As we defined in (\ref{e4.27}), we have
\begin{align}\label{e4.29}
\int_{{\Omega_\varepsilon}\cap B(\xi_0,r)}M(x)=\sum_{l=1}^k\int_{A_l}M(x).
\end{align}
Then we first compute $\int_{A_l}M(x)$ for any $l=1,\cdots,k$. Fix $l$, we get
\begin{align}\label{e4.30}
\int_{A_l}M(x)=&\int_{A_l}\left(|\sum_{j=1}^k(-1)^{j+1}PU_{j}|^{p+1}-PU_{l}^{p+1}-(p+1)PU_{l}^{p}\sum_{i\neq l}(-1)^{i+l}PU_{i}\right)
\nonumber\\
&-\sum_{i\neq l}\int_{A_l}PU_{i}^{p+1}-(p+1)\int_{A_l}\left(\sum_{i>j}(-1)^{i+j} PU_{i}^p PU_{j}-PU_{l}^{p}\sum_{i\neq l}(-1)^{i+l}PU_{i}\right)
\nonumber\\
=&(p+1)\sum_{j>l}(-1)^{l+j}\int_{A_l}U_{l}^p U_{j}
\nonumber\\
&+\int_{A_l}\left(|\sum_{j=1}^k(-1)^{j+1}PU_{j}|^{p+1}-PU_{l}^{p+1}-(p+1)PU_{l}^{p}\sum_{i\neq l}(-1)^{i+l}PU_{i}\right)
\nonumber\\
&-\sum_{i\neq l}\int_{A_l}PU_{i}^{p+1}
+(p+1)\sum_{j>l}(-1)^{l+j}\int_{A_l}\bigg((PU_{l}^p-U_{l}^p )U_{j}+PU_{l}^p(PU_{j}-U_{j})\bigg)
\nonumber\\
&+(p+1)\sum_{i>j,i\neq l}(-1)^{i+j}\int_{A_l} PU_{i}^p PU_{j}
\nonumber\\
:=&(p+1)\sum_{j>l}(-1)^{l+j}\int_{A_l}U_{l}^p U_{j}+M_1+M_2+M_3+M_4.
\end{align}
Let $x-\xi_0=\mu_{l\varepsilon}z$, then $A_l$ becomes $\tilde{A}_l=\left\{z\in \mathbb{R}^N: \sqrt{\frac{\mu_{l+1\varepsilon}}{\mu_{l\varepsilon}}}\leq |z|\leq \sqrt{\frac{\mu_{l-1\varepsilon}}{\mu_{l\varepsilon}}}\right\}$ for $l=1,\cdots,k-1$,
then when $j>l$ and $l=1,\cdots,k-1$, we get
\begin{align}\label{e4.32}
\int_{A_l}U_{l}^pU_{j}=&\alpha_N^{p+1}\int_{A_l}\left(\frac{{\mu_{l\varepsilon}}}{{\mu_{l\varepsilon}}^2+|x-\xi_{l\varepsilon}|^2}\right)^{\frac{N+4}{2}}\left(\frac{{\mu_{j\varepsilon}}}{{\mu_{j\varepsilon}}^2+|x-\xi_{j\varepsilon}|^2}\right)^{\frac{N-4}{2}}
\nonumber\\
=&\alpha_N^{p+1}\left(\int_{\mathbb{R}^N}\frac{1}{|z|^{N-4}(1+|z-\sigma_l|^2)^{\frac{N+4}{2}}}\right)\left(\frac{\mu_{j}}{\mu_{l}}\right)^{\frac{N-4}{2}}\varepsilon^{\frac{(N-4)(j-l)\theta}{2k}}(1+o(1)).
\end{align}
Therefore,
\begin{align}\label{e4.33}
\sum_{j>l}(-1)^{l+j}\int_{A_l} U_{l}^p U_{j}=-\alpha_N^{p+1}\Gamma(\sigma_l)\left(\frac{\mu_{l+1}}{\mu_{l}}\right)^{\frac{N-4}{2}}\varepsilon^{\frac{N-4}{2k}\theta}(1+o(1)).
\end{align}
Via a Taylor expansion, we get
\begin{align*}
\left|M_1\right|\leq c\left(\sum_{j\neq l}\int_{A_l}U_l^{p-1}U_j^{2}+\sum_{i,j\neq l}\int_{A_l}U_i^{p-1}U_j^{2}\right).
\end{align*}
Using H\"{o}lder inequality and Lemma \ref{lem6.6}, we have
\begin{align*}
\int_{A_l}U_l^{p-1}U_j^{2}\leq c\left(\int_{A_l}U_l^{p}U_j\right)^{\frac{p-1}{p}}\left(\int_{A_l}U_j^{p+1}\right)^{\frac{1}{p}}
\leq c\varepsilon^{\frac{(N-4)\theta}{2k}(1+\frac{4}{N+4})},
\end{align*}
and
\begin{align*}
\int_{A_l}U_i^{p-1}U_j^{2}\leq c\left(\int_{A_l}U_i^{p}U_j\right)^{\frac{p-1}{p}}\left(\int_{A_i}U_j^{p+1}\right)^{\frac{1}{p}}\leq c\varepsilon^{\frac{N\theta}{2k}}.
\end{align*}
Then we get $\left|M_1\right|\leq C\varepsilon^{\frac{(N-4)\theta}{2k}(1+\frac{4}{N+4})}$.

By Lemma \ref{lem6.6}, we get $\left|M_2\right|=O(\varepsilon^{\frac{N\theta}{2k}})$ and
\begin{align*}
\left|M_4\right|\leq c\sum_{i>j,i\neq l}\int_{A_l} U_{i}^pU_{j}\leq c\left(\sum_{i>j,j\neq l}\int_{A_l} U_{i}^pU_{j}+\sum_{i>l}\int_{A_l} U_{i}^pU_{l}\right)\leq c\varepsilon^{\frac{N\theta}{2k}}.
\end{align*}
By Lemma \ref{lem6.1} and Lemma \ref{lem6.7}, we have
\begin{align*}
\left|M_3\right|\leq c\sum_{j>l}\left(\int_{A_l}(PU_{l}^p-U_{l}^p )U_{j}+U_{l}^p(PU_{j}-U_{j})\right)
=O\left(\varepsilon^{\frac{N-4}{k}\theta}\right).
\end{align*}
Therefore
\begin{align}\label{e4.34}
J_R=\alpha_N^{p+1}\sum_{l=1}^{k-1}\Gamma(\sigma_l)\left(\frac{\mu_{l+1}}{\mu_{l}}\right)^{\frac{N-4}{2}}\varepsilon^{\frac{N-4}{2k}\theta}(1+o(1)).
\end{align}
By Lemma \ref{lem6.5} and (\ref{e4.34}), we get (\ref{e4.16}).

Next, we give the $C^1$ estimate of reduction energy $J_\varepsilon(V)$. Let $\partial_s$ denote $\partial_{\mu_j}$ for $j= 1,\cdots,k$ and $\partial_{\sigma^i_r}$ for $r=1,\cdots,k-1$ and $i=1,\cdots,N$.
It holds:
\begin{align*}
\partial_s\left(J_\varepsilon(V)\right)=\partial_s\left(\sum_{j=1}^kJ_{\varepsilon}(PU_{j})+J_R\right)=\partial_s\left(\sum_{j=1}^kJ_{\varepsilon}(PU_{j})\right)+\partial_s(J_R).
\end{align*}
For $\partial_s=\partial_{\mu_j}$, $j=2,\cdots,k$, we have
\begin{align*}
\partial_s(J_R)=&\partial_{\mu_j}\left(\alpha_N^{p+1}\sum_{l=1}^{k-1}\Gamma(\sigma_l)\left(\frac{\mu_{l+1}}{\mu_{l}}\right)^{\frac{N-4}{2}}\varepsilon^{\frac{N-4}{2k}\theta}(1+o(1))\right)
\nonumber\\
=&\frac{N-4}{2}\alpha_N^{p+1}\left(\Gamma(\sigma_{j-1})\left(\frac{\mu_{j}}{\mu_{j-1}}\right)^{\frac{N-4}{2}}\varepsilon^{\frac{N-4}{2k}\theta}-\Gamma(\sigma_{j})\left(\frac{\mu_{j+1}}{\mu_{j}}\right)^{\frac{N-4}{2}}\varepsilon^{\frac{N-4}{2k}\theta}\right)(1+o(1)).
\end{align*}
Then by Lemma \ref{lem6.5}, we have $\partial_{\mu_j}\left(J_\varepsilon(V)\right)=\partial_{\mu_j}\left(\Phi(\mu,\sigma)\right)+o(\varepsilon^{\frac{N-4}{2k}\theta})$.
In an analogous way, for $j=1,\cdots,N$ and $r=1,\cdots,k-1$, we have
\begin{align*}
\partial_{\sigma_r^j}\left(J_\varepsilon(V)\right)=\partial_{\sigma_r^j}\left(\Phi(\mu,\sigma)\right)+o(\varepsilon^{\frac{N-4}{2k}\theta}).
\end{align*}
\end{proof}

\noindent {\it Proof of Proposition \ref{pro3.4} (ii)}:
From Lemma \ref{lem4.2} and Lemma \ref{thm4.3}, we deduced that $I(\mu, \sigma)$ is $C^0-$uniform over all $\mu$, $\sigma$ satisfying constraints (\ref{e2.8}). We next prove the $C^1$ estimate of $I(\mu, \sigma)=J(V+\phi)$.

Let $\partial_s$ denote $\partial_{\mu_j}$ for $j= 1,\cdots,k$ and $\partial_{\sigma^i_r}$ for $r=1,\cdots,k-1$ and $i=1,\cdots,N$. It holds:
\begin{align}\label{zong1}
\partial_sJ(V+\phi)-\partial_s\Phi(\mu,\sigma)=&J'(V+\phi)[\partial_sV+\partial_s\phi]-\partial_s\Phi(\mu,\sigma)
\nonumber\\
=&[J'(V+\phi)-J'(V)][\partial_sV]+J'(V+\phi)[\partial_s\phi]+[\partial_sJ(V)-\partial_s\Phi(\mu,\sigma)]
\nonumber\\
:=&J_1+J_2+J_3.
\end{align}

{\it Step 1}. We claim  $J_1=o(\varepsilon^{\frac{(N-4)\theta}{2k}})$.

Since
\begin{align}\label{zong11.0}
J_1=-\int_{\Omega_\varepsilon}\left[f(V+\phi)-f(V)+f'(V)\phi\right]\partial_sV-\int_{\Omega_\varepsilon}f'(V)\phi\partial_sV
:=J_{11}+J_{12}.
\end{align}

We first estimate $J_{11}$ for the case $N\geq 13$, then
\begin{align}\label{zong11.1}
|J_{11}|\leq &c\mu_{i\varepsilon}\sum_{h}\int_{\Omega_\varepsilon}|f(V+\phi)-f(V)+f'(V)\phi|\left(|PZ_i^h-Z_i^h|+|Z_i^h|\right)
\nonumber\\
\leq &c\mu_{i\varepsilon}\sum_{h}|\phi|^p_{\frac{2N}{N-4}}|PZ_i^h-Z_i^h|_{\frac{2N}{N-4}}+c\int_{\Omega_\varepsilon}|f(V+\phi)-f(V)+f'(V)\phi|U_i.
\end{align}
By Proposition \ref{pro3.1} and $\mu_{i\varepsilon}|PZ_i^h-Z_i^h|_{\frac{2N}{N-4}}=O(\mu_{i\varepsilon}^{\frac{N-4}{2}})$, we obtain
\begin{align}\label{zong11.2}
\mu_{i\varepsilon}|\phi|^p_{\frac{2N}{N-4}}|PZ_i^h-Z_i^h|_{\frac{2N}{N-4}}=O(\varepsilon^{\frac{(N-4)\theta}{2k}\frac{p^2}{2}+\frac{(N-4)\theta}{4k}})=O(\varepsilon^{\frac{(N-4)\theta}{2k}}).
\end{align}
Moreover
\begin{align}\label{zong11.3}
\int_{\Omega_\varepsilon}|f(V+\phi)-f(V)+f'(V)\phi|U_i
=&\int_{\Omega_\varepsilon\setminus B(\xi_0,\rho)}|f(V+\phi)-f(V)+f'(V)\phi|U_i\nonumber\\
&+\left(\int_{A_i}+\sum_{l\neq i}^k\int_{A_l}\right)|f(V+\phi)-f(V)+f'(V)\phi|U_i,
\end{align}
where
\begin{align}\label{zong11.4}
&\int_{\Omega_\varepsilon\setminus B(\xi_0,\rho)}|f(V+\phi)-f(V)+f'(V)\phi|U_i\nonumber\\
\leq & c \mu_{i\varepsilon}^{\frac{N-4}{2}}\int_{\Omega_\varepsilon\setminus B(\xi_0,\rho)}|\phi|^p\leq c \mu_{i\varepsilon}^{\frac{N-4}{2}}|\phi|^p_{\frac{2N}{N-4}}=o(\varepsilon^{\frac{N-4}{2k}\theta}),
\end{align}
and for $i=l$, we have
\begin{align}\label{zong11.5}
&\int_{A_i}|f(V+\phi)-f(V)+f'(V)\phi|U_i\nonumber\\
\leq& \int_{A_i}|f(V+\phi)-f(V)-f'(U_i)\phi|U_i+|f'(U_i)-f'(V)||\phi|U_i.
\end{align}
Then using H\"{o}lder inequality and Lemma \ref{lem7.0}, we get
\begin{align}\label{zong11.6}
\int_{A_i}|f'(U_i)-f'(V)||\phi|U_i\leq & |[f'(U_i)-f'(V)]U_i|_{\frac{2N}{N+4}}|\phi|_{\frac{2N}{N-4}}\nonumber\\
=&o(\varepsilon^{\frac{N-4}{2k}\frac{\theta p}{2}+\frac{N-4}{2k}\frac{\theta p}{2}})=o(\varepsilon^{\frac{N-4}{2k}\theta}).
\end{align}
By the mean value Theorem, we obtain, for $t\in [0,1]$,
\begin{align}\label{zong11.7}
&\int_{A_i}|f(V+\phi)-f(V)+f'(U_i)\phi|U_i= \int_{A_i}|f'(V+t\phi)-f'(U_i)||\phi|U_i
\nonumber\\
\leq &c\int_{A_i}U_i^{p-2}\left|\sum_{j\neq i}PU_j+PU_i-U_i+t\phi\right||\phi|U_i
\nonumber\\
\leq &c\sum_{j\neq i}|U_i^{p-1}U_j|_{\frac{2N}{N+4}}|\phi|_{\frac{2N}{N-4}}+c|PU_i-U_i|_{\frac{2N}{N-4}}|U_i^{p-1}|_{\frac{N}{4}}|\phi|_{\frac{2N}{N-4}}+c|U_i^{p-1}|_{\frac{N}{4}}|\phi|^2_{\frac{2N}{N-4}}
\nonumber\\
=&o(\varepsilon^{\frac{N-4}{2k}\theta}).
\end{align}
For $i\neq l$, by H\"{o}lder inequality and Proposition \ref{pro3.1}, we deduce that
\begin{align}\label{zong11.8}
\int_{A_l}|f(V+\phi)-f(V)+f'(V)\phi|U_i\leq c\int_{A_l}|\phi|^pU_i\leq c|\phi|^p_{\frac{2N}{N-4}}|U_i|_{\frac{2N}{N-4}}=o(\varepsilon^{\frac{N-4}{2k}\theta}).
\end{align}
Then by (\ref{zong11.1})-(\ref{zong11.8}), we have
\begin{align}\label{zong11.9}
|J_{11}|=o(\varepsilon^{\frac{N-4}{2k}\theta}) \quad \text{ for } N\geq 13.
\end{align}

We next consider the case $4<N\leq 12$. Using $|V|_{\frac{2N}{12-N}}=O(\varepsilon^{\frac{2\theta}{k}})$ and $|\partial_sV|_{\frac{2N}{N-4}}=O(1)$, then
\begin{align}\label{zong11.10}
|J_{11}|\leq |V|_{\frac{2N}{12-N}} |\phi|^2_{\frac{2N}{N-4}} |\partial_sV|_{\frac{2N}{N-4}}+ |\phi|^p_{\frac{2N}{N-4}} |\partial_sV|_{\frac{2N}{N-4}}=o(\varepsilon^{\frac{N-4}{2k}\theta}).
\end{align}
From (\ref{zong11.9}) and (\ref{zong11.10}), we get
\begin{align}\label{zong11}
|J_{11}|=o(\varepsilon^{\frac{N-4}{2k}\theta}).
\end{align}
Moreover,
\begin{align}\label{zong12}
|J_{12}|=&\left|\int_{\Omega_\varepsilon}f'(V)\phi\partial_sV\right|\leq \mu_{i\varepsilon}\sum_{j}\left(\left|\int_{\Omega_\varepsilon}f'(V)\phi(PZ_i^j-Z_i^j)\right|+\left|\int_{\Omega_\varepsilon}[f'(V)-f'(U_i)]\phi Z_i^j)\right|\right)
\nonumber\\
\leq &\mu_{i\varepsilon}\sum_{j}|V^p|_{\frac{N}{4}}|\phi|_{\frac{2N}{N-4}}|PZ_i^j-Z_i^j|_{\frac{2N}{N-4}}+\sum_{j}|\phi|_{\frac{2N}{N-4}}|[f'(V)-f'(U_i)]U_i|_{\frac{2N}{N+4}}=o(\varepsilon^{\frac{N-4}{2k}\theta}).
\end{align}
Then by (\ref{zong11}) and (\ref{zong12}), it follows that $J_{1}=o(\varepsilon^{\frac{N-4}{2k}\theta})$.

{\it Step 2}. We show that $J_{2} =o(\varepsilon^{\frac{N-4}{2k}\theta})$.

By (\ref{e2.9}), we have
\begin{align}\label{j22}
J_{2}=J'(V+\phi)[\partial_s\phi]=\sum_{l=0}^{N}\sum_{j=1}^{k}c^l_j\langle PZ^l_j, \partial_s\phi \rangle.
\end{align}
According to (\ref{e3.46}) and (\ref{e3.47}), we know that
\begin{align}\label{j2.1}
\langle PZ^l_j, \partial_s\phi \rangle=\begin{cases}
  0, &\text{ if }s=\mu_h, h\neq i, h=1,\cdots,k,\\
  O(\frac{\|\phi\|}{\mu_{h\varepsilon}}), &\text{ if }s=\mu_h, h=i,\\
  0, &\text{ if }s=\sigma_r^j, r\neq i, r=1,\cdots,k-1, j=1,\cdots,N,\\
  O(\frac{1}{\mu_{r\varepsilon}}), &\text{ if }r=i.\end{cases}
\end{align}
Moreover, by (\ref{e2.9}), we get $J'(V+\phi)[PZ^t_h]=\sum_{l=0}^{N}\sum_{j=1}^{k}c^l_j\langle PZ^l_j, PZ^t_h \rangle$.

Since $\phi \in K^{\perp}$, we see
\begin{align*}
J'(V+\phi)[PZ^t_h]
=&\int_{\Omega_\varepsilon}\left[\sum_{i=1}^{k}(-1)^{i+1}f(U_i)-f(V)\right]PZ^t_h+\int_{\Omega_\varepsilon}\left[f(V)-f(V+\phi)\right]PZ^t_h
\nonumber\\
=&O\left(|\sum_{i=1}^{k}(-1)^{i+1}f(U_i)-f(V)|_{\frac{2N}{N+4}}|PZ^t_h|_{\frac{2N}{N-4}}\right)
\nonumber\\
&+O\left(|V^{p-1}|_{\frac{N}{4}}|\phi|_{\frac{2N}{N-4}}|PZ^t_h|_{\frac{2N}{N-4}}\right)+O\left(|\phi|^p_{\frac{2N}{N-4}}|PZ^t_h|_{\frac{2N}{N-4}}\right).
\end{align*}
From the estimate of $W_1$ in Lemma \ref{lem4.4}, $|PZ^t_h|_{\frac{2N}{N-4}}=O(\frac{1}{\mu_{h\varepsilon}})$ and $|V^{p-1}|_{\frac{N}{4}}=O(1)$, we obtain
\begin{align*}
J'(V+\phi)[PZ^t_h]\leq \begin{cases}
  c\frac{\varepsilon^{\frac{(N-4)\theta}{2k}\frac{p}{2}}}{\mu_{h\varepsilon}}, &\text{ if }N\geq 13,\\
  c\frac{\varepsilon^{\frac{(N-4)\theta}{2k}}|\ln \varepsilon|}{\mu_{h\varepsilon}}, &\text{ if }N=12,\\
  c\frac{\varepsilon^{\frac{(N-4)\theta}{2k}}}{\mu_{h\varepsilon}}, &\text{ if }5\leq N\leq 11.\end{cases}
\end{align*}
Then by Lemma \ref{lem6.3}, we get
\begin{align}\label{j2.2}
c^l_j=\begin{cases}O(\varepsilon^{\frac{(N-4)\theta}{2k}\frac{p}{2}}), &\text{ if }N\geq 13,\\
  O(\varepsilon^{\frac{(N-4)\theta}{2k}}|\ln \varepsilon|), &\text{ if }N=12,\\
  O(\varepsilon^{\frac{(N-4)\theta}{2k}}), &\text{ if }5\leq N\leq 11.\end{cases}
\end{align}
Thus from (\ref{j22}), (\ref{j2.1}) and (\ref{j2.2}), we have $J_{2}=J'(V+\phi)[\partial_s\phi]=o(\varepsilon^{\frac{N-4}{2k}\theta})$.

{\it Step 3}. We estimate $J_{3}=o(\varepsilon^{\frac{N-4}{2k}\theta})$. This estimate is proved in Lemma \ref{thm4.3}.

Therefore, we get the $C^1$ estimate of $I(\mu, \sigma)=J(V+\phi)$.
\qed

\section{Appendix}\label{se8}

In the Appendix, we give some estimates.

\begin{lemma}\label{lem6.2}
Let $d>0$ small but fixed and assume $\mu$, $\sigma$ satisfying (\ref{e2.8}). Then
\begin{equation}\label{e6.5}
\int_{\Omega_\varepsilon}U_{\mu,\xi}^{p-1}(P_\varepsilon U_{\mu,\xi}-U_{\mu,\xi})^2=\begin{cases}O(\mu^N+\frac{\varepsilon^{N}}{\mu^N}), \ \  N\geq 9,\\
O(\mu^8|\log \mu|+\frac{\varepsilon^{12}}{\mu^{12}}|\log (\frac{\varepsilon}{\mu})|), \ \ N=8,\\
O(\mu^{2(N-4)}+\frac{\varepsilon^{2(N-2)}}{\mu^{2(N-2)}}), \ \ 5 \leq N<8.
\end{cases}
\end{equation}
\end{lemma}

\begin{proof}
Using Lemma \ref{lem6.1}, we have
\begin{equation*}
0\leq U_{\mu,\xi}-P_\varepsilon U_{\mu,\xi}\leq C\left(\mu^{\frac{N-4}{2}}+\frac{\varepsilon^{N-2}}{\mu^{\frac{N}{2}}}\frac{1}{|x-\xi_0|^{N-4}}+\frac{\varepsilon^{N-2}}{\mu^{\frac{N-4}{2}}}\frac{1}{|x-\xi_0|^{N-2}}+\frac{\varepsilon^{N}}{\mu^{\frac{N}{2}}}\frac{1}{|x-\xi_0|^{N-2}}\right).
\end{equation*}
Thus it is equivalent to evaluate
\begin{equation*}
\int_{\Omega_\varepsilon}\left(\frac{\mu}{\mu^2+|x-\xi|^2}\right)^4\left(\mu^{N-4}+\frac{\varepsilon^{2(N-2)}}{\mu^{N}|x-\xi_0|^{2(N-4)}}+\frac{\varepsilon^{2(N-2)}}{\mu^{N}|x-\xi_0|^{2(N-2)}}+\frac{\varepsilon^{2N}}{\mu^{N}|x-\xi_0|^{2(N-2)}}\right).
\end{equation*}

(i) If $N\geq 9$, then $\int_{\Omega_\varepsilon}\left(\frac{\mu}{\mu^2+|x-\xi|^2}\right)^4=O\left(\mu^4 \int_{\Omega}\frac{1}{|x-\xi_0|^8}\right)$.\\
Set $x-\xi_0=\varepsilon y$, then $\int_{\Omega_\varepsilon}\left(\frac{\mu}{\mu^2+|x-\xi|^2}\right)^4\frac{1}{|x-\xi_0|^{2(N-4)}}=O\left(\varepsilon^{-(N-8)}\mu^{-4} \int_{|y|\geq 1}\frac{1}{|y|^{2(N-4)}}\right)$
and
\begin{equation*}
\int_{\Omega_\varepsilon}\left(\frac{\mu}{\mu^2+|x-\xi|^2}\right)^4\frac{1}{|x-\xi_0|^{2(N-2)}}=O\left(\varepsilon^{-(N-4)}\mu^{-4} \int_{|y|\geq 1}\frac{1}{|y|^{2(N-2)}}\right).
\end{equation*}

(ii) If $N=8$, setting $x-\xi=\mu y$, then $\int_{\Omega_\varepsilon}\left(\frac{\mu}{\mu^2+|x-\xi|^2}\right)^4=O\left(\mu^4 |\log {\mu}|\right)$,\\
\begin{equation*}
\int_{\Omega_\varepsilon}\left(\frac{\mu}{\mu^2+|x-\xi|^2}\right)^4\frac{1}{|x-\xi_0|^{2(N-4)}}=O\left(\mu^{-4} |\log (\frac{\varepsilon}{\mu})|\right)
\end{equation*}
and
\begin{equation*}
\int_{\Omega_\varepsilon}\left(\frac{\mu}{\mu^2+|x-\xi|^2}\right)^4\frac{1}{|x-\xi_0|^{2(N-2)}}=O\left(\mu^{-8} \int_{\mathbb{R}^N}\frac{1}{(1+|y|^2)^4}\frac{1}{|y+\sigma|^{2(N-2)}}\right).
\end{equation*}

(iii) If $5\leq N<8$, setting $x-\xi=\mu y$, then $\int_{\Omega_\varepsilon}\left(\frac{\mu}{\mu^2+|x-\xi|^2}\right)^4=O\left(\mu^4 \int_{\mathbb{R}^N}\frac{1}{(1+|y|^2)^4}\right)$,\\
\begin{equation*}
\int_{\Omega_\varepsilon}\left(\frac{\mu}{\mu^2+|x-\xi|^2}\right)^4\frac{1}{|x-\xi_0|^{2(N-4)}}=O\left(\mu^{-N+4}\int_{\mathbb{R}^N}\frac{1}{(1+|y|^2)^4}\frac{1}{|y+\sigma|^{2(N-4)}} \right)
\end{equation*}
and
\begin{equation*}
\int_{\Omega_\varepsilon}\left(\frac{\mu}{\mu^2+|x-\xi|^2}\right)^4\frac{1}{|x-\xi_0|^{2(N-2)}}=O\left(\mu^{-N} \int_{\mathbb{R}^N}\frac{1}{(1+|y|^2)^4}\frac{1}{|y+\sigma|^{2(N-2)}}\right).
\end{equation*}
Collecting all the previous estimates, the claim follows.
\end{proof}

By direct calculation, as \cite[Lemma A.5]{Musso}, we have the following estimates.
\begin{lemma}\label{lem6.3}
It holds:
\begin{align*}
\langle PZ^j_i, PZ^h_l\rangle=o\left(\frac{1}{\mu_{i\varepsilon}^2}\right), \text{if}\  l>i,
\end{align*}
\begin{align*}
\langle PZ^j_i, PZ^h_i\rangle=o\left(\frac{1}{\mu_{i\varepsilon}^2}\right), \text{if}\  j\neq h,
\end{align*}
\begin{align*}
\langle PZ^j_i, PZ^j_i\rangle=\frac{c_j}{\mu_{i\varepsilon}^2}(1+o(1)),
\end{align*}
for some positive constants $c_0$ and $c_1=\cdots=c_N$.
\end{lemma}

\begin{lemma}\label{lem6.6}
The following estimates hold:
\begin{align}\label{e6.14}
\int_{A_l}U_{j}^{p+1}=O(\varepsilon^{\frac{N\theta}{2k}}) \quad \mbox{for\ all}\ j\neq l,
\end{align}
\begin{align}\label{e6.15}
\int_{A_l}U_{i}^p U_{j}=O(\varepsilon^{\frac{N\theta}{2k}}) \quad \mbox{for\ all}\ j\neq l, i\neq l,
\end{align}
\begin{align}\label{e6.16}
\int_{A_l}U_{l}^p U_{j}=O(\varepsilon^{\frac{(N-4)|l-j|\theta}{2k}}) \quad \mbox{for\ all}\ j\neq l.
\end{align}
\end{lemma}
\begin{proof}
To get (\ref{e6.14}), we perform the change of variable $x-\xi_0=\mu_{j\varepsilon}z$, then $A_l$ becomes
$$\hat{A}_l=\left\{z\in \mathbb{R}^N: \frac{\sqrt{\mu_{l+1\varepsilon}\mu_{l\varepsilon}}}{\mu_{j\varepsilon}}\leq |z|\leq \frac{\sqrt{\mu_{l\varepsilon}\mu_{l-1\varepsilon}}}{\mu_{j\varepsilon}}\right\}.
$$
If $j>l$, then $\frac{\sqrt{\mu_{l\varepsilon}\mu_{l-1\varepsilon}}}{\mu_{j\varepsilon}}\rightarrow \infty$ and $\left|\int_{A_l}U_{j}^{p+1}\right|=\left|\int_{\hat{A}_l}\frac{1}{(1+|z-\sigma_j|^2)^{N}}\right|
\leq  c\int_{\frac{\sqrt{\mu_{l\varepsilon}\mu_{l-1\varepsilon}}}{\mu_{j\varepsilon}}}^{\infty}t^{-N-1}
=O(\varepsilon^{\frac{N\theta}{2k}})$.\\
If $j<l$, then $\frac{\sqrt{\mu_{l\varepsilon}\mu_{l+1\varepsilon}}}{\mu_{j\varepsilon}}\rightarrow 0$ and $\left|\int_{A_l}U_{j}^{p+1}\right|=\left|\int_{\hat{A}_l}\frac{1}{(1+|z-\sigma_j|^2)^{N}}\right|
\leq c\left(\frac{\sqrt{\mu_{l\varepsilon}\mu_{l-1\varepsilon}}}{\mu_{j\varepsilon}}\right)^N=O(\varepsilon^{\frac{N\theta}{2k}})$.\\
To get (\ref{e6.15}). Using H\"{o}lder inequality and (\ref{e6.14}), we get
\begin{align*}
\left|\int_{A_l}U_{i}^p U_{j}\right| \leq & c\left(\int_{A_l}U_{i}^{p+1}\right)^{\frac{p}{p+1}}\left(\int_{A_l}U_{j}^{p+1}\right)^{\frac{1}{p+1}}
=O(\varepsilon^{\frac{N\theta}{2k}}).
\end{align*}
To get (\ref{e6.16}). Considering $j<l$, then
\begin{align*}
\int_{A_l}U_{l}^p U_{j}=&\alpha_N^{p+1}\int_{\tilde{A}_l}\left(\frac{\mu_{l\varepsilon}}{\mu_{j\varepsilon}}\right)^{\frac{N-4}{2}}\left(\frac{1}{1+|z-\sigma_{l}|^2}\right)^{\frac{N+4}{2}}\left(\frac{1}{1+|\frac{\mu_{l\varepsilon}}{\mu_{j\varepsilon}}z-\sigma_{j}|^2}\right)^{\frac{N-4}{2}}
\nonumber\\
=&\alpha_N^{p+1}\left(\int_{\mathbb{R}^N}\frac{1}{(1+|\sigma_j|^2)^{\frac{{N-4}}{2}}(1+|z-\sigma_l|^2)^{\frac{N+4}{2}}}\right)\left(\frac{\mu_{l}}{\mu_{j}}\right)^{\frac{N-4}{2}}\varepsilon^{\frac{(N-4)(l-j)\theta}{2k}}(1+o(1)).
\end{align*}
From the above equality and (\ref{e4.32}), we get (\ref{e6.16}).\\
\end{proof}

\begin{lemma}\label{lem6.7}
It holds:
\begin{align}\label{x2.1}
\int_{A_l}U_{l}^p(PU_{j}-U_{j})=O\left(\varepsilon^{\frac{N-4}{k}\theta}\right)\ \  \mbox{and} \ \ \int_{A_l}(PU_{l}^p-U_{l}^p )U_{j}=O\left(\varepsilon^{\frac{N-4}{k}\theta}\right).
\end{align}
\end{lemma}
\begin{proof}
By Lemma \ref{lem6.1}, we get
\begin{align*}
\left|\int_{A_l}U_{l}^p(PU_{j}-U_{j})\right|\leq &c\int_{A_l}U_{l}^p \bigg({\mu_{j\varepsilon}}^{\frac{N-4}{2}}+\frac{\varepsilon^{N-2}}{{\mu_{j\varepsilon}}^{\frac{N}{2}}}\frac{1}{|x-\xi_0|^{N-4}}
\nonumber\\
&\qquad\quad\quad+\frac{\varepsilon^{N-2}}{{\mu_{j\varepsilon}}^{\frac{N-4}{2}}}\frac{1}{|x-\xi_0|^{N-2}}+\frac{\varepsilon^{N}}{{\mu_{j\varepsilon}}^{\frac{N}{2}}}\frac{1}{|x-\xi_0|^{N-2}}\bigg).
\end{align*}
Since
\begin{align*}
\int_{A_l}U_{l}^p{\mu_{j\varepsilon}}^{\frac{N-4}{2}}\leq &
c\left(\mu_{j\varepsilon}\mu_{l\varepsilon}\right)^{\frac{N-4}{2}}\int_{\tilde{A}_l}\left(\frac{1}{1+|z-\sigma_l|^2}\right)^{\frac{N+4}{2}}
\leq c {\varepsilon}^{\frac{2(j+l)-2}{2k}\frac{N-4}{2}\theta}
\leq c {\varepsilon}^{\frac{N-4}{k}\theta},
\end{align*}
and in an analogy way, we get $\int_{A_l}U_{l}^p\frac{\varepsilon^{N-2}}{{\mu_{j\varepsilon}}^{\frac{N}{2}}}\frac{1}{|x-\xi_0|^{N-4}}=O({\varepsilon}^{\frac{N-4}{k}\theta})$
and
\begin{align*}
\int_{A_l}U_{l}^p\frac{\varepsilon^{N-2}}{{\mu_{j\varepsilon}}^{\frac{N-4}{2}}}\frac{1}{|x-\xi_0|^{N-2}}=O({\varepsilon}^{\frac{N-4}{k}\theta}),\ \int_{A_l}U_{l}^p\frac{\varepsilon^{N}}{{\mu_{j\varepsilon}}^{\frac{N}{2}}}\frac{1}{|x-\xi_0|^{N-2}}=O({\varepsilon}^{\frac{N-4}{k}\theta}).
\end{align*}
Thus, we deduce that $\int_{A_l}U_{l}^p(PU_{j}-U_{j})=O\left(\varepsilon^{\frac{N-4}{k}\theta}\right)$.

Similarly, we can also get
\begin{align*}
\left|\int_{A_l}(PU_{l}^p-U_{l}^p )U_{j}\right|\leq & c\int_{A_l}U_{l}^{p-1}|PU_{l}-U_{l}|U_{j}
\leq c{\varepsilon}^{\frac{N-4}{k}\theta}.
\end{align*}
\end{proof}

\begin{lemma}\label{lem6.8}
It holds:
\begin{align*}
\int_{A_l}|U_l^{p-1}U_i|^{\beta}\leq
\begin{cases}c\varepsilon^{\frac{(N-4)\theta}{2k}\frac{p\beta}{2}}, &\text{ if }N\geq 13,\\
  c\varepsilon^{\frac{(N-4)\theta}{2k}\beta} |\ln \varepsilon|^\beta, &\text{ if }N=12,\\
  c\varepsilon^{\frac{(N-4)\theta}{2k}\beta}, &\text{ if }5 \leq N \leq 11,\end{cases}
\end{align*}
where $\beta=\frac{2N}{N+4}$.
\end{lemma}
\begin{proof}
If $N>12$, set $x-\xi_0=\mu_{i\varepsilon}$, then
\begin{align*}
&\int_{A_l}|U_l^{p-1}U_i|^{\beta}\leq  c\int_{\frac{\sqrt{\mu_{l\varepsilon}\mu_{l+1\varepsilon}}}{\mu_{i\varepsilon}} \leq|y|\leq \frac{\sqrt{\mu_{l\varepsilon}\mu_{l-1\varepsilon}}}{\mu_{i\varepsilon}}}\frac{\mu_{l\varepsilon}^{4\beta}}{(\mu_{l\varepsilon}^{2}+\mu_{i\varepsilon}^{2}|y-\frac{\mu_{l\varepsilon}}{\mu_{i\varepsilon}}\sigma_l|^2)^{4\beta}}\frac{\mu_{i\varepsilon}^{N-\frac{N-4}{2}\beta}}{(1+|y-\sigma_i|^2)^{\frac{N-4}{2}\beta}}
\nonumber\\
=&
\begin{cases}O\left((\frac{\mu_{l\varepsilon}}{\mu_{i\varepsilon}})^{4\beta}\right)\int_{\frac{\sqrt{\mu_{l\varepsilon}\mu_{l+1\varepsilon}}}{\mu_{i\varepsilon}} \leq|y|\leq \frac{\sqrt{\mu_{l\varepsilon}\mu_{l-1\varepsilon}}}{\mu_{i\varepsilon}}}\frac{1}{|y-\frac{\mu_{l\varepsilon}}{\mu_{i\varepsilon}}\sigma_l|^{8\beta}}\frac{1}{(1+|y-\sigma_i|^2)^{\frac{N-4}{2}\beta}}, &\text{ if }l>i,\\
  O\left((\frac{\mu_{i\varepsilon}}{\mu_{l\varepsilon}})^{4\beta}\right)\int_{\frac{\sqrt{\mu_{l\varepsilon}\mu_{l+1\varepsilon}}}{\mu_{i\varepsilon}} \leq|y|\leq \frac{\sqrt{\mu_{l\varepsilon}\mu_{l-1\varepsilon}}}{\mu_{i\varepsilon}}}\frac{1}{(1+|y-\sigma_i|^2)^{\frac{N-4}{2}\beta}}, &\text{ if }l<i,\end{cases}
\nonumber\\
=&O\left(\varepsilon^{\frac{N\theta}{2k}}\right)
=O\left(\varepsilon^{\frac{(N-4)\theta}{2k}\frac{p\beta}{2}}\right).
\end{align*}
Similarly, if $5\leq N<12$, set $x-\xi_0=\mu_{l\varepsilon}$, we have
\begin{align*}
&\int_{A_l}|U_l^{p-1}U_i|^{\beta}\leq  c\int_{\frac{\sqrt{\mu_{l\varepsilon}\mu_{l+1\varepsilon}}}{\mu_{l\varepsilon}} \leq|y|\leq \frac{\sqrt{\mu_{l\varepsilon}\mu_{l-1\varepsilon}}}{\mu_{l\varepsilon}}}\frac{\mu_{l\varepsilon}^{N-4\beta}}{(1+|y-\sigma_l|^2)^{4\beta}}\frac{\mu_{i\varepsilon}^{N-\frac{N-4}{2}\beta}}{(\mu_{i\varepsilon}^2+\mu_{l\varepsilon}^2|y-\frac{\mu_{i\varepsilon}}{\mu_{l\varepsilon}}\sigma_i|^2)^{\frac{N-4}{2}\beta}}
\nonumber\\
=&
\begin{cases}O\left((\frac{\mu_{l\varepsilon}}{\mu_{i\varepsilon}})^{\frac{N-4}{2}\beta}\right)\int_{\frac{\sqrt{\mu_{l\varepsilon}\mu_{l+1\varepsilon}}}{\mu_{l\varepsilon}} \leq|y|\leq \frac{\sqrt{\mu_{l\varepsilon}\mu_{l-1\varepsilon}}}{\mu_{l\varepsilon}}}\frac{1}{(1+|y-\sigma_l|^2)^{4\beta}}, &\text{ if }l>i,\\
  O\left((\frac{\mu_{i\varepsilon}}{\mu_{l\varepsilon}})^{\frac{N-4}{2}\beta}\right)\int_{\frac{\sqrt{\mu_{l\varepsilon}\mu_{l+1\varepsilon}}}{\mu_{l\varepsilon}} \leq|y|\leq \frac{\sqrt{\mu_{l\varepsilon}\mu_{l-1\varepsilon}}}{\mu_{l\varepsilon}}}\frac{1}{|y-\frac{\mu_{i\varepsilon}}{\mu_{l\varepsilon}}\sigma_i|^{(N-4)\beta}}\frac{1}{(1+|y-\sigma_l|^2)^{4\beta}}, &\text{ if }l<i,\end{cases}
\nonumber\\
=&O\left(\varepsilon^{\frac{(N-4)\theta}{2k}\beta}\right).
\end{align*}
If $N=12$, we get that $\int_{A_l}|U_l^{p-1}U_i|^{\beta}=O\left(\varepsilon^{\frac{(N-4)\theta}{2k}\beta} |\ln \varepsilon|^\beta\right).$
\end{proof}

\begin{lemma}\label{lem7.0}
The following estimate holds:
\begin{align*}
\left|[f'(U_i)-f'(V)]U_i\right|_{\frac{2N}{N+4}}=
\begin{cases}
  c\varepsilon^{\frac{(N-4)\theta}{2k}\frac{p}{2}}, &\text{ if }N\geq 13,\\
  c\varepsilon^{\frac{(N-4)\theta}{2k}}|\ln \varepsilon|, &\text{ if }N=12,\\
  c\varepsilon^{\frac{(N-4)\theta}{2k}}, &\text{ if }5\leq N\leq 11.\end{cases}
\end{align*}
\end{lemma}
\begin{proof}
Set $U_{\mu_{i\varepsilon},\xi_{i\varepsilon}}:=U_i$. It holds:
\begin{align}\label{z1}
\int_{\Omega_\varepsilon}\left|[f'(U_i)-f'(V)]U_i\right|^{\frac{2N}{N+4}}=\left(\int_{\Omega_\varepsilon\setminus B(\xi_0,\rho)}+\int_{A_i}+\sum_{l\neq i}^{k}\int_{A_l}\right)\left|[f'(U_i)-f'(V)]U_i\right|^{\frac{2N}{N+4}},
\end{align}
where
\begin{align}\label{f1}
\int_{\Omega_\varepsilon\setminus B(\xi_0,\rho)}\left|[f'(U_i)-f'(V)]U_i\right|^{\frac{2N}{N+4}}\leq c\left|\bigg(\sum_{ j=1}^{k}\mu_{j\varepsilon}^4\bigg)\mu_{i\varepsilon}^{\frac{N-4}{2}}\right|^{\frac{2N}{N+4}}=O(\varepsilon^{\frac{N\theta}{2k}}).
\end{align}
Next, let us consider the case $N\geq 13$. For $l\neq i$, we have
\begin{align}\label{f2}
\int_{A_l}\left|[f'(U_i)-f'(V)]U_i\right|^{\frac{2N}{N+4}}\leq & c\int_{A_l}\Big||(PU_i-U_i)+\sum_{j\neq i}PU_j|^{p-1}U_i\Big|^{\frac{2N}{N+4}}
\nonumber\\
\leq & c\int_{A_l}\left||PU_i-U_i|^{p-1}U_i\right|^{\frac{2N}{N+4}}+c\sum_{j\neq i}\int_{A_l}\left|U_j^{p-1}U_i\right|^{\frac{2N}{N+4}}.
\end{align}
It holds
\begin{align}\label{f21}
\int_{A_l}\left||PU_i-U_i|^{p-1}U_i\right|^{\frac{2N}{N+4}}\leq c\int_{A_l}U_i^{\frac{2N}{N+4}}\mu_{i\varepsilon}^{\frac{N-4}{2}(p-1)\frac{2N}{N+4}}=O(\varepsilon^{\frac{(N-4)\theta}{2k}\frac{p\beta}{2}}),
\end{align}
where $\beta=\frac{2N}{N+4}$.
As for $\sum_{j\neq i}\int_{A_l}\left|U_j^{p-1}U_i\right|^{\frac{2N}{N+4}}$.

If $l\neq i$ and $j=l$. The estimate of $\int_{A_l}|U_l^{p-1}U_i|^{\beta}$ is a direct result of Lemma \ref{lem6.8}.

If $l\neq i$ and $j\neq l$. Using H\"{o}lder inequality, we have
\begin{align}\label{f23}
\int_{A_l}|U_j^{p-1}U_i|^{\beta}\leq \left(\int_{A_l}U_i^{p+1}\right)^{\frac{N-4}{N+4}}\left(\int_{A_l}U_j^{p+1}\right)^{\frac{8}{N+4}}=O(\varepsilon^{\frac{N\theta}{2k}})=O(\varepsilon^{\frac{(N-4)\theta}{2k}\frac{p\beta}{2}}).
\end{align}
Thus by Lemma \ref{lem6.8}, (\ref{f2})-(\ref{f23}), we deduce
\begin{align*}
\int_{A_l}\left|[f'(U_i)-f'(V)]U_i\right|^{\frac{2N}{N+4}}=O(\varepsilon^{\frac{(N-4)\theta}{2k}\frac{p\beta}{2}}) \text{ for } l\neq i.
\end{align*}
If $l=i$, we obtain
\begin{align}\label{f3}
\int_{A_i}\left|[f'(U_i)-f'(V)]U_i\right|^{\frac{2N}{N+4}}\leq c\int_{A_i}\left||PU_i-U_i|U_i^{p-1}\right|^{\frac{2N}{N+4}}+ c\sum_{j\neq i}\int_{A_i}\Big|U_i^{p-1}U_j\Big|^{\frac{2N}{N+4}}.
\end{align}
As the proof of Lemma \ref{lem6.7}, we get
\begin{align*}
\int_{A_i}\left||PU_i-U_i|U_i^{p-1}\right|^{\frac{2N}{N+4}}=O(\varepsilon^{\frac{(N-4)\theta}{2k}\frac{p\beta}{2}}).
\end{align*}
Therefore, it follows that $\left|[f'(U_i)-f'(V)]U_i\right|_{\frac{2N}{N+4}}=O(\varepsilon^{\frac{(N-4)\theta}{2k}\frac{p}{2}})$ by Lemma \ref{lem6.8}.  Similarly, we can get the results for $5\leq N\leq 11$ and $N=12$.
\end{proof}

{\bf Acknowledgments:}
The research has been supported by National Natural Science Foundation of China 11971392, Natural Science
Foundation of Chongqing, China cstc2021ycjh-bgzxm0115, and Fundamental Research Funds for the
Central Universities XDJK2020B047.

\end{document}